\documentclass[reqno,centertags,11pt,a4paper]{amsart}
\usepackage{amssymb,mathrsfs}
\usepackage{color,umoline}
\usepackage[dvipsnames]{xcolor}
\usepackage{graphicx}
\usepackage{enumitem}
\usepackage[font=small,labelfont=bf]{caption}
\usepackage{tikz, caption, subcaption}
\usetikzlibrary{quotes,angles}
\usepackage{relsize}
\usepackage[mathscr]{eucal}
\usepackage{verbatim}
\usepackage[draft]{todonotes}
\usepackage{hyperref}
\usepackage[dvipsnames]{xcolor}
\usetikzlibrary{patterns}
\tikzstyle{EDR}=[draw=lightgray,line width=0pt,preaction={clip, postaction={pattern=north east lines, pattern color=gray}}]
\tikzstyle{EDR1}=[draw=lightgray,line width=0pt,preaction={clip, postaction={pattern=north west lines, pattern color=gray}}]

\newcommand{\dist}{\operatorname{dist}}

\definecolor{mygray}{gray}{0.95}

\definecolor{mypink1}{rgb}{1.2,1.1,0.9}

\definecolor{mypink2}{rgb}{1.0,0.95 ,0.9}

\definecolor{mypink3}{rgb}{1.0,0.6,0.7}

\numberwithin{equation}{section}

\newtheorem{theorem}{Theorem}[section]
\newtheorem{definition}[theorem]{Definition}
\newtheorem{lemma}[theorem]{Lemma}
\newtheorem{corollary}[theorem]{Corollary}
\newtheorem{conjecture}[theorem]{Conjecture}
\newtheorem{proposition}[theorem]{Proposition}
\newtheorem{example}[theorem]{Example}
\theoremstyle{remark}
\newtheorem{remark}{Remark}

\numberwithin{equation}{section}

\newcommand{\R}{\mathbb R}

\newcommand{\dif}{\mathrm{d}}
\newcommand{\beq}{\begin{equation}}
	\newcommand{\eeq}{\end{equation}}
\newcommand{\beqq}{\begin{equation*}}
	\newcommand{\eeqq}{\end{equation*}}
\newcommand{\ben}{\begin{eqnarray}}
	\newcommand{\een}{\end{eqnarray}}
\newcommand{\beno}{\begin{eqnarray*}}
	\newcommand{\eeno}{\end{eqnarray*}}

\def \f{\frac}

\begin{document}

	\title[A type  of oscillatory integral operator and its applications]
	{A type  of oscillatory integral operator and its applications}

	\author[C. Gao]{Chuanwei Gao}
	\address{Beijing International Center for Mathematical Research, Peking University, Beijing, China}
	\email{cwgao@pku.edu.cn}
	
	\author[J.Li]{Jingyue Li}
\address{Institute of Applied Physics and Computational Mathematics, Beijing 100088, P. R. China}
\email{m\_lijingyue@163.com}

	\author[L. Wang]{Liang Wang}
	\address{
		School of Mathematics and Statistics\\
		Wuhan University\\
		Wuhan 430072\\
		China}
	\email{wlmath@whu.edu.cn}

	

	\begin{abstract}
		In this paper, we consider  $L^p$- estimate for a class of oscillatory integral operators satisfying the Carleson-Sj\"olin conditions with further convex  and straight assumptions. As applications, the multiplier problem related to a general class of hypersurfaces with nonvanishing Gaussian curvature, local smoothing estimates for the fractional Schr\"odinger equation  and the sharp resolvent estimates outside of the uniform boundedness range  are discussed.
	\end{abstract}
	
	\maketitle
	
\section{Introduction}	
Let $n \geq 2$, $a \in C_c^{\infty}(\R^n \times \R^{n-1})$ be non-negative and supported in $B^n_1(0) \times B^{n-1}_1(0)$ and $\phi \colon B^n_1(0) \times B^{n-1}_1(0) \to \R$ be a smooth function which satisfies the following Carleson-Sj\"olin conditions:
\begin{itemize}
	\item[H1)] $\mathrm{rank}\, \partial_{\xi x}^2 \phi(x,\xi) = n-1$ for all $(x,\xi) \in B^n _1(0)\times B^{n-1}_1(0)$;
	\item[H2)] Defining the map $G \colon B^n_1(0) \times B^{n-1}_1(0) \to S^{n-1}$ by $G(x,\xi) := \frac{G_0(x,\xi)}{|G_0(x,\xi) |}$ where
	\begin{equation*}
		G_0(x,\xi) := \bigwedge_{j=1}^{n-1} \partial_{\xi_j} \partial_x\phi(x,\xi),
	\end{equation*}
	the curvature condition
	\begin{equation*}
		\det \partial^2_{\xi \xi} \langle \partial_x\phi(x,\xi),G(x, \xi_0)\rangle|_{\xi = \xi_0} \neq 0
	\end{equation*}
	holds for all $(x, \xi_0) \in \mathrm{supp}\,a$.
	%Furthermore, we also assume that $ G_0(x,\xi)$ keeps invariant when $x$ changes.

\end{itemize}
For any $\lambda\geq 1$, define the operator $T^{\lambda}$ by
\begin{equation}\label{eq:00}
	T^{\lambda}f(x) :=  \int_{B^{n-1}_1(0)} e^{2 \pi i \phi^{\lambda}(x, \xi)}a^{\lambda}(x,\xi) f(\xi)\,d \xi
\end{equation}
where  $f \colon B^{n-1}_1(0) \to \mathbb{C}, a(x,\xi)\in C_c^\infty(B_{1}^{n}(0)\times B_1^{n-1}(0))$ and $$a^{\lambda}(x, \xi) := a(x/\lambda, \xi), \phi^{\lambda}(x,\xi) :=\lambda\phi(x/\lambda,\xi).$$
We say $T^\lambda$ is a H\"ormander type operator if $\phi$ satisfies the conditions H1) and H2).  A typical example for the H\"ormander-type operator is the following extension operator $E$ defined by 
\beq\label{eq:02a}
Ef(x):=\int_{B_1^{n-1}(0)} e^{2\pi i(x'\cdot \xi+x_n\psi(\xi))}f(\xi)d\xi,
\eeq
with 
\beqq
{\rm rank}\Big(\frac{\partial^2 \psi}{\partial \xi_i\partial \xi_j}\Big)_{(n-1)\times (n-1)}=n-1.
\eeqq
H\"ormander conjectured that  if $\phi$ satisfies conditions ${\rm H}_1,{\rm H}_2$, then 
\beq\label{eq:01a}
\|T^\lambda f\|_{L^p(\R^n)}\lesssim \|f\|_{L^p(B_1^{n-1}(0))},
\eeq 
for $p>\frac{2n}{n-1}$. H\"ormander \cite{Hor} proved the above conjecture for $n=2$. For the higher dimensional case, Stein \cite{stein1} proved \eqref{eq:01a}  for $p\geq 2\frac{n+1}{n-1}$ and  $n\geq 3$. Later, Bourgain \cite{Bourgain1} disproved H\"ormander's conjecture  by constructing a kind of counterexample. Furthermore, he showed that Stein's result is sharp in the odd dimensions. For the even dimensions, up to the endpoint case, Bourgain, Guth \cite{BG} proved the sharp result. In summary, we  may state  the results as follows.
\begin{theorem}[\cite{stein1},\cite{BG}]\label{theo1}
	Let $n\geq 3$ and  $T^\lambda$ be a H\"ormander type operator.   For all $\varepsilon>0,\lambda\geq 1$,
	\beq
	\|T^\lambda f\|_{L^p(\R^n)}\lesssim_{\varepsilon,\phi,a} \lambda^{\varepsilon}\|f\|_{L^p(B^{n-1}_1(0))}
	\eeq
	holds whenever 
	 \begin{equation}\label{eq:mainaa}p\geq \left\{\begin{aligned}
			&2\tfrac{n+1}{n-1}\quad \text{\rm for $n$ odd},\\
			&2\tfrac{n+2}{n}\quad \text{\rm for $n$ even}.
		\end{aligned}\right.\end{equation}
	\end{theorem}
 Lee \cite{LS} observed  that if we further impose the following  \emph{convex  condition}
\begin{itemize}
	\item[H3)] The eigenvalues of the Hessian 
	\begin{equation*}
		\partial^2_{\xi \xi} \langle \partial_x\phi(x,\xi),G(x, \xi_0)\rangle|_{\xi = \xi_0} 
	\end{equation*}
	are all positive for $(x, \xi_0) \in \mathrm{supp}\,a$;
	%Furthermore, we also assume that $ G_0(x,\xi)$ keeps invariant when $x$ changes.
	\end{itemize}
on the phase, the range of $p$  can be obtained beyond that in \eqref{eq:mainaa}. Recently, Guth-Hickman-Iliopoulou \cite{GHI} proved the sharp results for the operator $T^\lambda$ with a  convex  phase. To be more precise, they showed 
\begin{theorem}[\cite{GHI}]\label{theo2}
Let $n\geq 3$ and $T^\lambda$ be a H\"ormander type operator satisfying the convex condition.  For all $\varepsilon>0,\lambda\geq 1$,
	\beq
	\|T^\lambda f\|_{L^p(\R^n)}\lesssim_{\varepsilon,\phi,a} \lambda^{\varepsilon}\|f\|_{L^p(B^{n-1}_1(0))}
	\eeq
	holds whenever 
	\begin{equation}\label{eq:maina}p\geq \left\{\begin{aligned}
			&2\tfrac{3n+1}{3n-3}\quad \text{\rm for $n$ odd},\\
			&2\tfrac{3n+2}{3n-2}\quad \text{\rm for $n$ even}.
		\end{aligned}\right.\end{equation}
\end{theorem}
The primary difference between the translation invariant case \eqref{eq:02a} and \eqref{eq:00} is that the main contribution of $T^\lambda f$ may be  concentrated in  a small neighborhood of a lower dimensional submanifold  which features slightly  differently between the odd and even dimensions. However, such phenomena can not happen for the extension operator $E$ if the Kakeya conjecture holds. The difference between Theorem \ref {theo1} and \ref{theo2} arises from  the fact that in the convex setting, such concentration lies in an at least  $\lambda^{1/2}$ neighborhood of a submanifold which can be manifested by the transverse equidistribution property, while for the general phase, it can be further squeezed into an $1$-neighborhood of a submainifold. 

As one can see, the Kakeya compression phenomena prohibit the sharp range of $p$ in \eqref{eq:mainaa},\eqref{eq:maina} to be matched with the conjectured range $p>\frac{2n}{n-1}$. Therefore, it is natural to conjecture the potentially possible range of $p$ in \eqref{eq:01a} will be $p>\frac{2n}{n-1}$ if the Kakeya compression phenomena does not happen. A  probable way to preclude the Kakeya compression phenomena is to impose the following \emph{straight condition} on the phase.
\begin{itemize}
	\item[H4)]  For given $\xi$,  $ G(x,\xi)$ keeps invariant when $x$ changes.
	\end{itemize}
 Formally,  we may formulate the following conjecture.

\begin{conjecture}\label{coj}
	Let $n\geq 3$ and $T^{\lambda}$ be a H\"ormander type operator with the  straight condition. For all $\varepsilon >0$,  the estimate
	\begin{equation}\label{gee}
		\|T^{\lambda}f\|_{L^p(\R^n)} \lesssim_{\phi, a}\|f\|_{L^{p}(B^{n-1}_1(0))}
	\end{equation}
	holds uniformly for $\lambda \geq 1$ whenever $p>\frac{2n}{n-1}$. 

\end{conjecture}
Obviously, Conjecture \ref{coj} implies the restriction conjecture,  and thus the Kakeya conjecture. Furthermore, we may see later Conjecture \ref{coj} also has many other applications. For example, Conjecture  \ref{coj} implies the Bochner-Riesz conjecture related to a class of general hypersurfaces with nonvanishing Gaussian curvature  and the local smoothing conjecture for the fractional Schr\"odinger equation and the sharp resolvent estimates outside of the uniform boundedness range.

In this paper,  we prove certain $L^p$ estimate for  $T^\lambda$ being a H\"ormander type operator with  the convex  and straight conditions.   Let $p_{n}$ be an  exponent which will be defined in Section \ref{fp}.  We may state our main results as follows.
\begin{theorem}\label{theo}
	Let $n\geq 3$ and $T^{\lambda}$ be a H\"ormander type operator with the convex and straight conditions. For all $\varepsilon >0$ the estimate
	\begin{equation}\label{eq:101}
		\|T^{\lambda}f\|_{L^p(\R^n)} \lesssim_{\varepsilon, \phi, a} \lambda^{\varepsilon}\|f\|_{L^{p}(B^{n-1}_1(0))}
	\end{equation}
	holds uniformly for $\lambda \geq 1$ whenever $p>p_{n}$. 
	\end{theorem}
The proof of Theorem \ref{theo} relies on  the polynomial partitioning method which was  introduced by Guth\cite{Guth,Guth18} to handle the restriction problem. Since then, it has been also used to study the  pointwise convergence problem for the Schr\"odigner operator, Bochner-Riesz conjecture, Kakeya conjecture and local smoothing conjecture for the wave equation and  the fractional Schr\"odinger equation, one may refer to \cite{DGL,Wu,HRZ,GLMX,GOW} and references therein for more details. Technically speaking, the straight condition can not be kept under the change of variables in the spatial space. To overcome this obstacle, we need to work with a more general class of functions which satisfy the straight condition up to a diffeomorphism in the spatial variables. It should be noted that the proof of Theorem \ref{theo} is obtained by adapting the arguments  in \cite{GHI,GOW}. Thus  we only streamline the structure of the proof  when there are too many overlaps.

The rest of this paper is organized as follows: In Section \ref{app}, we will show the  applications of conjecture \ref{coj} to the multiplier problem, local smoothing estimates for the fractional Schr\"odinger equation and  the sharp resolvent estimates outside of the uniform boundedness range.  In Section \ref{reduc}, we perform some reductions. In particular, we introduce a special class of  functions to make the induction arguments completed. In Section \ref{wav}, we introduce the wave packet decomposition which is an important tool. In Section  \ref{tran}, we prepare some useful ingredients which play important roles in the proof of the broad ``norm" estimate  in Section \ref{broad}.  With the above preparations, finally, we prove Theorem \ref{theo}  in Section \ref{fp}.

{\bf Acknowledgment:}This project originates from  a summer school program held by Shaoming Guo, Bochen Liu and Yakun  Xi at Zhejiang University and is partially supported by NSF China grant NO. 12171424. We are grateful for all  organizers for providing  an opportunity to discuss the problem. The authors thank Prof.Shaoming Guo for constructive discussion and  valuable suggestions. The first author is grateful for Prof.Gang Tian's support.  C.Gao was supported by Chinese Postdoc Foundation Grant 8206300279.

 \indent{\bf Notations.} For nonnegative quantities $X$ and $Y$, we will write $X\lesssim Y$ to denote the inequality $X\leq C Y$ for some $C>0$. If $X\lesssim Y\lesssim X$, we will write $X\sim Y$. Dependence of implicit constants on the spatial dimensions or integral exponents such as $p$ will be suppressed; dependence on additional parameters will be indicated by subscripts. For example, $X\lesssim_u Y$ indicates $X\leq CY$ for some $C=C(u)$. We write $A(R)\leq {\rm RapDec}(R) B$ to mean that for any power $\beta$, there is a constant $C_\beta$ such that
\beqq
| A(R)|\leq C_\beta R^{-\beta}B \,\,\quad\text{for all}\,\; R\geq 1.
\eeqq
We will also often abbreviate $\|f\|_{L_x^r(\R^n)}$ to $\|f\|_{L^r}$.  For $1\leq r\leq\infty$,
we use $r'$ to denote the dual exponent to $r$ such that $\tfrac{1}{r}+\tfrac{1}{r'}=1$. Throughout the paper, $\chi_E$ is the characteristic function of the set $E$.
We usually denote by $B_r^n(a)$ a ball in $\R^n$ with center $a$ and radius $r$. We will also denote by $B_R^n$ a ball of radius $R$ and arbitrary center in $\R^n$.  Denote by $A(r):=B_{2r}^n(0)\setminus B_{r/2}^n (0)$. We denote $w_{B^{n}_{R}(x_0)}$ to be a nonnegative weight function adapted to the ball $B^{n}_{R}(x_0)$ such that
$$ w_{B^{n}_{R}(x_0)}(x)\lesssim (1+R^{-1}|x-x_0|)^{-M},$$
for some large constant $M\in \mathbb{N}$.

We define the Fourier transform on $\mathbb{R}^n$ by
\begin{equation*}
	\aligned \hat{f}(\xi):= \int_{\mathbb{R}^n}e^{-2\pi  ix\cdot \xi}f(x)\,\dif x:=\mathcal{F}f(\xi).
	\endaligned
\end{equation*}
and the inverse Fourier transform by
\beqq
\check{g}(x):=\int_{\R^n} e^{2\pi ix\cdot \xi}g(\xi)\dif \xi:=(\mathcal{F}^{-1}g)(x).
\eeqq
These help us to define the fractional differentiation operators $\vert\nabla\vert^s$ and $\langle\nabla\rangle^s$ for $s\in \R$ via
$$\vert\nabla\vert^s f(x):=\mathcal{F}^{-1}\big\{\vert\xi\vert^s\hat{f}(\xi)\big\}(x)\quad \text{and}\quad \langle\nabla\rangle^s f(x):=
\mathcal{F}^{-1}\big\{ (1+|\xi|^2)^\frac{s}{2}\hat{f}(\xi)\big\}(x).$$
In this manner, we define the  Sobolev norm of the space $L^p_{\alpha}(\R^n)$  by
$$\|f\|_{L^p_{\alpha}(\R^n)}:=\big\|\langle\nabla\rangle^\alpha f\big\|_{L^p(\R^n)}.$$

\section{Applications}\label{app}
In this section, we talk about the relations of Conjecture \ref{coj} to other associated  problems. In particular, the multiplier problem with respect to a  general class of hypersurface with non-vanishing Gaussian curvature, local smoothing conjecture  for the fractional Schr\"odinger and the sharp resovent estimate outside of uniform boundedness range  will be discussed.

Let $\psi:\R^{n-1}\rightarrow \R$ be a smooth function with
\beqq
{\rm rank}\Big(\frac{\partial^2\psi}{\partial \xi_i\partial \xi_j}\Big)_{(n-1)\times (n-1)}=n-1,
\eeqq
and 
\beqq
|\partial^\alpha \psi(\xi)|\leq  1, \quad \alpha \in \mathbb{Z}^{n-1}, |\alpha|\leq N,
\eeqq
where $N$ is a large constant.
Therefore, by inverse function theorem,  there exists  locally  a function $g:\R^{n-1}\rightarrow \R^{n-1}$ such that 
\beq\label{eq:002}
\partial_{\xi}\psi(g(x'))=-x',\quad x'\in \R^{n-1}.
\eeq
\subsection{Multiplier problem}\label{sub1}
Let $\delta\geq 0,\xi=(\xi',\xi_n)$ and  $m^\delta(\xi):=(\xi_n-\psi(\xi'))^\delta_{+}\chi(\xi')$, where $\chi$ is a smooth compactly supported function with ${\rm supp} \chi \subset B_2^{n-1}(0)$  and
\beqq
t^{\delta}_{+}=\left\{\begin{array}{ccc}
	t^\delta,\quad t\geq 0,\\
	0,\quad t<0.
\end{array}\right.
\eeqq
We consider the following multiplier problem: for which $\delta$ and $p$ such that 
\beq\label{eq:b1}
\big\|m^\delta(D)f\big\|_{L^p(\R^n)}\lesssim_\delta \|f\|_{L^p(\R^n)}.
\eeq
It is conjectured that 
\begin{conjecture}\label{coj1}
	For $\delta\geq 0$ and $1\leq p\leq \infty$, then 
	\beq
	\big\|m^\delta(D)f\big\|_{L^p(\R^n)}\lesssim_{\delta}\|f\|_{L^p(\R^n),\quad} \delta>\delta(p):=\max\Big\{n\big|\frac{1}{2}-\frac{1}{p}\big|-\frac{1}{2}, 0\Big\}.
	\eeq
\end{conjecture}

We will show how Conjecture \ref{coj} implies Conjecture \ref{coj1}.  Let $p>\frac{2n}{n-1}$ and $\eta:\R\rightarrow \R$ be a smooth compactly supported function, with ${\rm supp} \eta \subset (1/2,1)$ satisfying 
\beqq
\sum_{j\in \mathbb{Z}}\eta(2^j t)\equiv 1,\quad t>0.
\eeqq
We break $m^\delta$ into pieces
\beqq
m^\delta(\xi)=\sum_{j\geq 1}\eta(2^j(\xi_n-\psi(\xi')))m^\delta(\xi)+r(\xi),
\eeqq
where $r(\xi)$ is a smooth function with ${\rm supp}r\subset B_2^n(0)$.

Define an operator $m_j^\delta(D)$ as follows:
\beqq
m^\delta_j(D)f(x):=\Big(\eta(2^j(\xi_n-\psi(\xi')))m^\delta(\xi)\hat{f}(\xi)\Big)^{\vee}(x).
\eeqq
Let $K_j^\delta(x)$  be the kernel of the multiplier $m_j^\delta(D)$, i.e.
\beqq
K_j^\delta(x)=\int_{\R^n}e^{2\pi ix \cdot \xi}\eta(2^j(\xi_n-\psi(\xi')))m^\delta(\xi)d\xi.
\eeqq 
Through changing of variables, we may reformulate  $K_j^\delta(x)$  as follows:
\beqq
K_j^\delta(x)=2^{-j\delta}\int_{\R} e^{2\pi ix_n \xi_n}\widetilde{\eta}(2^j\xi_n) \int_{\R^{n-1}} e^{2\pi i(x'\cdot \xi'+x_n\psi(\xi'))}\chi(\xi')d\xi' d\xi_n,
\eeqq
where 
\beqq
\widetilde{\eta}(t)=\eta(t)t^\delta_{+}.
\eeqq
For convenience, define 
\beqq
K_j(x):=\int_{\R} e^{2\pi ix_n \xi_n}\widetilde{\eta}(2^j\xi_n) \int_{\R^{n-1}} e^{2\pi i(x'\cdot \xi'+x_n\psi(\xi'))}\chi(\xi')d\xi' d\xi_n.
\eeqq
To handle  the inner part of the integral with respect to $\xi'$, we use the stationary phase method. For this purpose, we borrow the following lemma from \cite{KL} with a slight modification. One may refer to \cite{KL} for the proof.
\begin{lemma}\label{station}
	Define 
	$$I_\psi(x):=\int_{\R^{n-1}} e^{2\pi i(x'\cdot \xi'+x_n\psi(\xi'))}\chi(\xi')d\xi',$$
	then 
	\begin{itemize}
		\item  If $|x_n|\geq 1/2$ and $2^5|x'|\leq |x_n|$, then for every $M\in \mathbb{N}$ satisfying $2M\leq N$ we have 
		\beq
		I_{\psi}(x)=\frac{c_n}{\sqrt{|K|}}e^{2\pi i(x'\cdot g(\frac{x'}{x_n})+x_n\psi(g(\frac{x'}{x_n})))}
		\times \sum_{j=0}^{M-1} \mathcal{D}_j \chi(\xi')|_{\xi'=g(\frac{x'}{x_n})}|x_n|^{-\frac{n-1}{2}-j}+\mathcal{E}_M(x),	\eeq
		where $c_n$ is a constant depending on $n$, $K$ denotes the Gaussian curvature of the hypersurface $(\xi',\psi(\xi'))$ at point $(g(\frac{x'}{x_n}),\psi(g(\frac{x'}{x_n})))$, $\mathcal{D}_0 \chi=\chi$ and $\mathcal{D}_j$ is a differential operator in $\xi'$ of order $2j$. For $\mathcal{E}$, we have the estimate
		\beqq
		|\mathcal{E}_M(x)|\lesssim_{M,\psi} |x_n|^{-M}.
		\eeqq
	\end{itemize}
	\item If $2^6|x'|\geq |x_n|$ or $|x_n|\leq 2$, then for every $0\leq M\leq N$ there exists a constant $C_M$, such that 
	\beqq
	|I_\psi(x)|\leq C_M(1+|x|)^{-M}.
	\eeqq	
\end{lemma}
Let $\widetilde \chi\in C_c^\infty(\R)$ with ${\rm supp}\widetilde \chi \subset (-2^{-5},2^{-5})$ equaling  to $1$ in $(-2^{-6},2^{-6})$,  $\beta\in C_0^\infty(\R)$ with ${\rm supp} \beta\subset [-9/8,-3/8]\cup [3/8,9/8]$ and 
\beqq
\sum_{\ell=-\infty}^\infty \beta(2^{-\ell}t)=1,\; t\neq 0.
\eeqq
We split  $K_j$ as follows:
\beqq
K_j(x)=K_{j,0}(x)+\sum_{\ell \geq 1}K_{j,\ell}(x),
\eeqq
where 
 \beqq
	K_{j,\ell}(x)=\widetilde\chi\big(\frac{|x'|}{x_n}\big)\beta(2^{-\ell}x_n)K_j(x).
	\eeqq
Using Lemma  \ref{station}, we have 
\beqq
\|K_{j,0}\ast f\|_{L^p}\lesssim \|f\|_{L^p}.
\eeqq
For $\ell\geq 1$, the main contribution to $K_{j,\ell}$ comes from $\widetilde{K}_{j,\ell}$ defined by 
\beqq
\widetilde{K}_{j,\ell}(x)=\widetilde{\chi}(\frac{|x'|}{x_n})
\beta(2^{-\ell}x_n)|x_n|^{-\frac{n-1}{2}}e^{2\pi i(x'\cdot g(\frac{x'}{x_n})+x_n\psi(g(\frac{x'}{x_n})))}\times \int_\R e^{2\pi ix_n\xi_n}\widetilde{\eta}(2^j\xi_n)d\xi_n.
\eeqq
Thus it suffices to show 
\beqq
\|\widetilde{T}_{j,\ell} f\|_{L^p(\R^n)}\lesssim 2^{(\frac{n+1}{2}-\frac{n}{p})\ell}2^{-j}(1+2^{\ell-j})^{-M}\|f\|_{L^p},
\eeqq
where $\widetilde{T}_{j,\ell}$ is defined  by
\begin{align*}
	\widetilde{T}_{j,\ell} f(x)=\int_{\R^n}\widetilde{K}_{j,\ell}(x-y)f(y)dy.
\end{align*}
Then by a standard optimization argument, we have 
\beq
\sum_{\ell=1}^\infty \|\tilde T_{j,\ell}f\|_{L^p}\lesssim 2^{(\frac{n-1}{2}-\frac{n}{p})j}\|f\|_{L^p}.
\eeq
Therefore, by a localization argument, it suffices to show 
\beqq
\|2^{\ell n}\widetilde{K}_{j,\ell}(2^{\ell}\cdot )\ast f\|_{L^p(B_1^{n}(0))}\lesssim 2^{(\frac{n+1}{2}-\frac{n}{p})\ell}2^{-j}(1+2^{\ell-j})^{-M}\|f\|_{L^p(B_1^n(0))}.
\eeqq
Note that 
\beq
\begin{aligned}
	2^{\ell n}\widetilde{K}_{j,\ell}(2^{\ell}\cdot)\ast f&=2^{\frac{n+1}{2}\ell }\int_{\R^n} e^{2\pi i2^\ell((x'-y')\cdot g(\frac{x'-y'}{x_n-y_n})+(x_n-y_n)\psi(g(\frac{x'-y'}{x_n-y_n})))}a_{\ell,j}(x,y)f(y)dy\\
	&=2^{\frac{n+1}{2}\ell } \int_{\R} T^{2^\ell}_{y_n} f_{y_n}dy_n,
\end{aligned}
\eeq
where 
\beqq
\begin{aligned}
	a_{\ell,j}(x,y):=\widetilde{\beta}(x_n-y_n)\widetilde\chi\Big(\frac{x'-y'}{x_n-y_n}\Big)\times \int_\R e^{2\pi i2^{\ell}(x_n-y_n)\xi_n}\widetilde{\eta}(2^j\xi_n)d\xi_n,\; \widetilde{\beta}(t):=\beta(t)|t|^{-\frac{n-1}{2}}
\end{aligned}
\eeqq
and 
\beqq
T_{y_n}^{2^\ell}f_{y_n}(x):=\int_{\R^{n-1}} e^{2\pi i2^\ell((x'-y')\cdot g(\frac{x'-y'}{x_n-y_n})+(x_n-y_n)\psi(g(\frac{x'-y'}{x_n-y_n})))}a_{\ell,j}(x,y)f(y',y_n)dy', \; f_{y_n}(\cdot):=f(\cdot, y_n).
\eeqq
It's easy to show that 
	\beq\label{eq:001}
	|\partial_x^\alpha a_{\ell,j}(x,y)|\lesssim_{\alpha,M} 2^{-j}(1+2^{j-\ell})^M.
	\eeq
	Indeed, since $\partial^{\alpha}_x\Big(\widetilde{\beta}(x_n-y_n)\widetilde\chi\Big(\frac{x'-y'}{x_n-y_n}\Big)\Big)$ is bounded for any $|\alpha|\geq 0$, it suffices to show 
	\beq\label{eq:001-add}
	\Big|\partial_{x_n}^\alpha\int_\R e^{2\pi i2^{\ell}(x_n-y_n)\xi_n}\widetilde{\eta}(2^j\xi_n)d\xi_n\Big|\lesssim_{\alpha,M} 2^{-j}(1+2^{j-\ell})^M.
	\eeq
	By integration by parts, \eqref{eq:001-add} follows easily. 

For fixed $y_n$, by changing of variables
\beqq
\frac{x'}{x_n-y_n}\rightarrow  x',\quad  \frac{1}{x_n-y_n}\rightarrow x_n,
\eeqq
under the new coordinates, the phase  $(x'-y')\cdot g\big(\frac{x'-y'}{x_n-y_n}\big)+(x_n-y_n)\psi\big(g\big(\frac{x'-y'}{x_n-y_n}\big)\big)$  becomes
\beq \label{phase1}
\Psi(x,y'):=\big(\frac{x'}{x_n}-y'\big)\cdot g(x'-x_ny')+\frac{1}{x_n}\psi(g(x'-x_ny')).
\eeq
A direct computation shows that the associated Gauss map $G(x,y')$ related to the hypersurface $\{\partial_x\Psi(x,y')\}$ at $y'$ is given by  
\beq
G(x,y')=\frac{(y',1)}{\sqrt{1+|y'|^2}}, 
\eeq
which is  obviously independent of the spatial variables $x$. Furthermore, $\Psi(x,y')$ satisfies the Carleson-Sj\"olin conditions by our assumption that the Hessian of $\psi$ is nondegerate. 

Recall \eqref{eq:001}, we may use Conjecture \ref{coj} to obtain that 
\beqq
\|T_{y_n}^{2^\ell}f_{y_n}\|_{L^p(B_1^n(0))}\lesssim 2^{-\frac{n\ell}{p}}2^{-j}(1+2^{\ell-j})^{-M}\|f_{y_n}\|_{L^p(B_1^n(0))},
\eeqq
uniformly for $y_n$. Finally, integrating with respect to $y_n$, we will obtain the desired results.

If we  impose an additional condition that all  eigenvalues of the Hessian of $\psi$ are positive, then $\Psi(x,y')$ also satisfies the convex condition. From the above discussion, as a direct consequence of Theorem \ref{theo}, we  also have 
\begin{corollary}
	Let $1\leq p\leq \infty, \psi:\R^{n-1}\longrightarrow \R$ be smooth and  $\Big(\frac{\partial^2\psi}{\partial \xi_i\partial \xi_j}\Big)_{(n-1)\times (n-1)}$ has  $(n-1)$ positive eigenvalues, then
	\beq
	\|m^\delta(D)f\|_{L^p(\R^n)}\lesssim_{\delta,\psi,p}\|f\|_{L^p(\R^n)}
	\eeq
	for all $p$ such that $\max\{p,p'\}>p_n$ and $\delta>\delta(p)$.
\end{corollary}

%A standard example is when $g(x)=-x$ and $\psi(x)=\frac{|x|^2}{2}$, then 
%$$\Psi(x,t,y)=\frac{|x-y|^2}{4t},$$
%which corresponds to the kernel of the Schrodinger operator $e^{it\Delta}$.

\subsection{Local smoothing estimates  for the fractional Schr\"odinger equation}
Let $u:\R^{n}\times \R\rightarrow \mathbb{C}$ be the solution to the following equation
\begin{equation}\label{eq-01}\left\{
	\begin{aligned}
		i\partial_tu+(-\Delta)^{\f{\alpha}{2}}u&=0,\quad (t,x)\in\R\times\R^n\\
		u(0,x)&=f(x),
	\end{aligned}\right.
\end{equation}
where $\alpha\in (0,1)\cup (1,\infty)$ and $f$ is a Schwartz function. The solution $u$ can be  expressed by
\begin{equation}\label{eq:e1}
	u(x,t)=e^{it(-\Delta)^{\f{\alpha}{2}}}f(x):=\int_{\R^n} e^{2\pi i(x\cdot\xi+t|\xi|^\alpha)}\hat{f}(\xi) d \xi.
\end{equation}
We are concerned with $L^p$-regularity of the solution $u$. For a fixed time $t$,   Fefferman and Stein \cite{FS}, Miyachi \cite{Miya2} showed the following optimal  $L^p$ estimate:
\beq \label{eq-08}
\|e^{it(-\Delta)^{\f{\alpha}{2}}}f\|_{L^p(\R^n)}\leq C_{t,p}\|f\|_{L^p_{s_{\alpha,p}}(\R^n)},\;\;\;\; s_{\alpha,p}:=\alpha n\Big|\frac{1}{2}-\frac{1}{p}\Big|, \;\; 1<p<\infty.
\eeq
It is conjectured that:
\begin{conjecture}[Local smoothing for the fractional Schr\"odinger operator]\label{coj-sch}
	Let $\alpha\in (0,1)\cup (1,\infty),p>2+\frac2n$ and $s\ge\alpha n(\tfrac{1}{2}-\tfrac{1}{p})-\f{\alpha}{p}$. Then
	\beq \label{eq-09}
	\|e^{it(-\Delta)^{\f{\alpha}{2}}}f\|_{L^p(\R^n\times [1,2])}\leq C_{p,s}\|f\|_{L^p_s(\R^n)}.
	\eeq
\end{conjecture}
We will show Conjecture \ref{coj} implies Conjecture \ref{coj-sch}. Indeed, following the reduction in \cite{GMZ}, up to the endpoint regularity, to show \eqref{eq-09}, it suffices to prove 
\beq\label{eq:main1}
\|e^{it\psi(D)} f\|_{L_{x,t}^p(B^{n}_{R^2}\times [R^2/2,R^2])}\lesssim_\varepsilon R^{2 n(\f{1}{2}-\f{1}{p})+\varepsilon}\|f\|_{L^p(\R^n)},\quad \; {\rm supp} ~\hat{f}\subset
B_1^n(0),
\eeq
where $\psi$ also satisfies 
\begin{itemize}
	\item $\psi(0)=0,\nabla\psi(0)=0$;
	\item  For $\xi_0 \in B_1^n(0)$, the absolute value of all eigenvalues of the Hessian $\big(\frac{\partial^2\psi}{\partial \xi_i\partial \xi_j}\big)|_{\xi=\xi_0}$  falls into $ [1/2,1)$.
\end{itemize}

By a localization argument,  we may also assume ${\rm supp}f\subset B_{R^2}^n$.
Note that 
\beqq
e^{it\psi(D)}f(x)=\int_{\R^n}\int_{\R^n}e^{2\pi i\big((x-y)\cdot \xi+t\psi(\xi)\big)}a(\xi)f(y)d\xi dy,
\eeqq
where $a\in C_c^\infty({B_2^n(0)})$. We denote $K(x,t,y)$ the kernel  of the operator $e^{it\psi(D)}$, then 
\beqq
K(x,t,y)=\int_{\R^n}e^{2\pi i((x-y)\cdot \xi+t\psi(\xi))}a(\xi)d\xi.
\eeqq
Through a standard stationary phase argument, we have 
\beqq
K(x,t,y)\sim |t|^{-\tfrac{n}{2}}e^{2\pi i\big((x-y)\cdot g(\tfrac{x-y}{t})+t\psi(g(\frac{x-y}{t}))\big)}a(\tfrac{x-y}{t}),
\eeqq
where $g$ is defined as in \eqref{eq:002} with $n-1$ being replaced by $n$.
Note $t\sim R^2$, therefore, it suffices to consider the following oscillatory integral operators
\beqq
R^{-n}\int_{\R^n}e^{2\pi i((x-y)\cdot g(\tfrac{x-y}{t})+t\psi(g(\frac{x-y}{t})))}a(\tfrac{x-y}{t})f(y)dy.
\eeqq
By changing of variables, 
$$x\rightarrow R^2 x, t\rightarrow R^2t, y\rightarrow R^2 y,$$
we have 
\beq
\begin{aligned}
	&\Big\|\int_{\R^n}e^{2\pi i\big((x-y)\cdot g(\tfrac{x-y}{t})+t\psi(g(\frac{x-y}{t}))\big)}a(\tfrac{x-y}{t})f(y)dy\Big\|_{L^p(B_{R^2}^{n}\times [R^2/2,R^2])}\\
	&\lesssim R^{2n+\tfrac{2(n+1)}{p}}\Big\|\int_{\R^n}e^{2\pi iR^2\big((x-y)\cdot g(\tfrac{x-y}{t})+t\psi(g(\frac{x-y}{t}))\big)}a(\tfrac{x-y}{t})f(R^2y)dy\Big\|_{L^p(B_{1}^{n}\times [1/2,1])}.
\end{aligned}
\eeq
Performing change of variables as follows
\beqq
\frac{x}{t}\rightarrow x,\quad \frac{1}{t}\rightarrow t,
\eeqq
the corresponding  phase under the new coordinates becomes 
\beqq
\Psi(x,t,y):=\big(\frac{x}{t}-y\big)g(x-ty)+\frac{1}{t}\psi(g(x-ty)).
\eeqq
By our assumption, $\Psi(x,t,y)$ satisfies the straight condition, as well as the Carleson-Sj\"olin conditions. Therefore, by Conjecture \ref{coj}, we have
\beq
\begin{aligned}
	\Big\|\int_{\R^n}e^{2\pi iR^2\big((x-y)\cdot g(\tfrac{x-y}{t})+t\psi(g(\frac{x-y}{t}))\big)}a(\tfrac{x-y}{t})f(R^2y)dy\Big\|_{L^p(B_{1}^{n}\times[1/2,1])}&\lesssim R^{-\tfrac{2(n+1)}{p}}\|f(R^2\cdot)\|_{L^p(B_1^{n})}\\
	&\lesssim R^{-\tfrac{2(n+1)}{p}-\tfrac{2n}{p}}\|f\|_{L^p(B_{R^2}^n)}.
\end{aligned}
\eeq
Thus, we complete the proof.

It should be noted that  Gan-Oh-Wu\cite{GOW} considered the local smoothing problem for the fractional Schr\"odinger equation via  a different approach  and mentioned  essentially the same method as above discussed. Furthermore,  it is possible  to further improve Gan-Oh-Wu's result by considering the H\"ormanger type operator with the convex and straight conditions using Wang's method \cite{Wang1} at least in dimension $n=2$.

\subsection{Sharp resolvent estimates outside of the uniform boundedness range }The resolvent estimate for the Laplatian is of the form
\beq\label{eq:02}
\|(-\Delta-z)^{-1}f\|_{L^q(\R^n)}\leq C(z,p,q)\|f\|_{L^p(\R^n)},\quad \forall z\in \mathbb{Z}\backslash [0,\infty).
\eeq
This inequality and its variants have  been applied to study the problems of uniform Sobolev estimates, unique continuation properties and limiting absorption principles, etc, one may refer to \cite{EKL,KSR,GS} for more  details.

Let's briefly review the results related to \eqref{eq:02}. In \cite{KSR}, Kenig, Ruiz and Sogge showed, for $z\in \mathbb{C}\backslash [0,\infty)$ and $$\frac{1}{p}-\frac{1}{q}=\frac{2}{n},\quad  \frac{2n}{n+3}<p<\frac{2n}{n+1},$$ with $n\geq 3$, the constant $C(p,q,n)>0$ can be obtained independent of $z$. By homogeneity,  a simple calculation shows that 
\beqq
\|(-\Delta-z)^{-1}\|_{p\rightarrow q}=|z|^{-1+\frac{n}{2}(\frac{1}{p}-\frac{1}{q})}\Big\|(-\Delta-\frac{z}{|z|})^{-1}\Big\|_{p\rightarrow q},\quad \forall z\in \mathbb{C}\backslash [0,\infty).
\eeqq
For $z\in \mathbb{S}^1\backslash \{1\}$, Guti$\acute{e}$rrez \cite{Guti} obtained the optimal range of  $p,q$ with $n\geq 3$ in the sense that the constant $C(p,q,n)$ is independent of $z$. To be more precise, if $z\in \mathbb{S}^1\backslash \{1\}$ and  $(\frac{1}{p}, \frac{1}{q})$ lies in the set 
\beqq
\Big\{(x,y): \frac{2}{n+1}\leq x-y\leq \frac{2}{n}, x>\frac{n+1}{2n}, y<\frac{n-1}{2n}\Big\},\quad n\geq 3,
\eeqq
the sharp constant $C(p,q,z)$ in \eqref{eq:02} can be obtained uniformly independent of $z$. 

To formally state results regarding $C(z,p,q)$ with $z\in \mathbb{S}^1\backslash \{1\}$ and  $(\frac1p,\frac1q)$ lying  outside of the uniform boundedness range,  let's firstly introduce some notations.
Let $I^2$ be a closed square defined by 
\beqq
I^2:=\{(x,y)\in \R^2: 0\leq x,y\leq 1\}.
\eeqq
For each $(x,y)\in I^2$, define 
\beqq
(x,y)':=(1-x,1-y).
\eeqq
Similarly, for any subset $\mathcal{R}\subset I^2$, define $\mathcal{R}'$ to be 
\beqq
\mathcal{R}':=\{(x,y)\in I^2:(x,y)'\in \mathcal{R}\}.
\eeqq
\begin{definition}
	For $X_1,\cdots,X_\ell \in I^2$, we denote by $[X_1,\cdots,X_\ell]$ the convex hull of the points $X_1,\cdots, X_\ell$. In particular, $[X,Y]$ will denote the closed line segment jointing $X$ and $Y$. We also denote by  $(X,Y)$ and $[X,Y)$ for the open interval $[X,Y]\backslash\{X,Y\}$ and the half-open interval $[X,Y]\backslash\{Y\}$ respectively.
\end{definition}
Set $C=(\frac12,\frac12)$ and 
\begin{align*}
	&B:=(\frac{n+1}{2n},\frac{(n-1)^2}{2n(n+1)}),\quad B':=(\frac{n^2+4n-1}{2n(n+1)},\frac{n-1}{2n}),\\
	&D:=(\frac{n-1}{2n},\frac{n-1}{2n}),\qquad \quad D':=(\frac{n+1}{2n},\frac{n+1}{2n}),\\
	&E:=(\frac{n+1}{2n},0),\qquad\qquad\quad  E':=(\frac{n-1}{2n},1),
\end{align*}
and 
\[\mathcal{R}_0=\mathcal{R}_0(n):= 
\begin{cases}
	\big\{ (x,y) : 0\le x, y\le 1, ~ 0\le x-y <1 \big\}
	&\text{ if } ~  n=2,\\[5pt]
	\big\{(x,y) : 0\le x, y\le 1,~ 0\le x-y \le  \frac2n\big\} \setminus \big\{ \big(1, \frac{n-2}n\big) , \big(\frac 2n, 0\big) \big\}	
	&\text{ if } ~  n\ge 3.
\end{cases}	\]
It is conjectured that:
\begin{conjecture}\label{conjecture}
	Let $n\geq 2$. If $(\frac1p,\frac1q)$ lies in $\mathcal{R}_0\backslash \big([B, E]\cup [B',E']\cup [D,C)\cup [D',C)\big)$,  
	%\footnote{See Figure \ref{fig1} and \ref{fig2}.} 
	then 
	\begin{equation}\label{con1}	
		\|(-\Delta-z)^{-1}\|_{p\to q}\simeq_{p,q,n} |z|^{-1+\frac n2(\frac 1p-\frac 1q)+\gamma_{p,q}}  \dist(z,[0,\infty))^{-\gamma_{p,q}}
	\end{equation}
	holds for $z\in\mathbb{C}\backslash [0,\infty)$,
	where ${\mathlarger \gamma}_{p,q}$ is defined as follows:
	\begin{equation}\label{def-gamma}
		\gamma_{p,q}: =\max \Big\{ 0, ~ 1-\frac{n+1}2\Big(\frac 1p-\frac 1q \Big),~ \frac{n+1}2-\frac np, ~ \frac nq-\frac{n-1}{2} \Big\}.
	\end{equation}
\end{conjecture}
\begin{remark}
	If $(\frac{1}{p},\frac{1}{q})\in \{B,B'\}$, the  restricted weak type estimate 
	\beq\label{eq:003}
	\|(-\Delta-z)^{-1}f\|_{q,\infty}\leq C|z|^{-1+\frac{n}{n+1}}\|f\|_{p,1}
	\eeq
	holds. One may refer to \cite{KL} for more details.
\end{remark}

From \cite{KL}, we know that the lower bounded of \eqref{con1} is true for all $n\geq 2$. For $n=2$, this conjecture has been completely established, one may refer to \cite{KL}. However, for $n\geq 3$, only partial positive results of Conjecture \ref{conjecture} have been proved, see Figure \ref{figthm}. More presicely, for $n\geq 3$, setting 
\[\mathcal{R}_1:=[P_*, P_0,C]\setminus\{C\}\]
Kown-Lee \cite{KL} showed Conjecture \ref{con1} holds except for $(\frac1p,\frac1q)\in \mathcal{R}_1\cup \mathcal{R}_1'$.  Among other things, following the proof of Proposition 4.1  in \cite{KL}, we have

\begin{figure}
	\captionsetup{type=figure,font=footnotesize}
	\centering
	\begin{tikzpicture} [scale=0.6]\scriptsize
		\path [fill=mypink3] (50/11, 40/11)--(15/4,15/4)--(10-15/4,10-15/4)--(10-40/11,10-50/11)--(5,5)node[above]{$C$};
		\path [fill=mypink1] (0,0)--(15/4,15/4)--(50/11, 40/11)--(5,5)--(10-40/11,10-50/11)--(10-15/4,10-15/4)--(10,10)--(10,6)--(4,0)--(0,0);
		\draw [<->] (0,10.7)node[above]{$y$}--(0,0) node[below]{$(0,0)$}--(10.7,0) node[right]{$x$};
		\draw (0,10) --(10,10)--(10,0) node[below]{$(1,0)$};
		%		\draw [dash pattern={on 2pt off 1pt},mypink1] (50/11, 40/11)--(15/4,15/4)--(5,5)--(10-15/4,10-15/4)--(10-40/11,10-50/11)--(5,5);
		\draw [dash pattern={on 2pt off 1pt}] (4,4)--(50/11, 40/11);
		\draw [dash pattern={on 2pt off 1pt}] (6,6)--(10-40/11, 10-50/11);
		\draw  (47/12,47/12)--(4,4);
		\draw  (10-47/12,10-47/12)--(6,6);
		\draw [mypink3] (4,4)--(15/4,15/4)--(50/11, 40/11)--(5,5)--(10-40/11, 10-50/11)--(10-15/4,10-15/4)--(6,6);
		\draw (50/11, 40/11)--(6,8/3)node[above]{$B$}--(10-8/3,4)node[left]{$B'$}--(10-40/11, 10-50/11);
		\draw [dash pattern={on 2pt off 1pt}] (4,4)node[above]{$D$}--(6,6)node[above]{$D'$};
		\draw (0,0)--(15/4,15/4);
		\draw (10-15/4,10-15/4)--(10,10);
		\draw (4,0)node[below]{$\frac 2n$}--(10,6);
		\draw [dash pattern={on 2pt off 1pt}]  (0,5)node[left]{$\frac12$}--(5,5)--(5,0)node[below]{$\frac12$}; 
		\draw [dash pattern={on 2pt off 1pt}] (6, 8/3)--(6, 2)--(6,0)node[below]{$E$};
		\draw [dash pattern={on 2pt off 1pt},] (10-8/3, 4)--(8, 4)--(10,4)node[right]{$E'$};
		%		\draw (4.5, 4) node{$\mathcal R_4$};
		%		\draw (6, 5.4) node{$\mathcal R'_4$};
		
		\draw (15/4,15/4) node[left]{$P_*$};
		\draw (50/11, 40/11)  node[below] {$P_0$};
		
		\draw [color=mypink3,fill=mypink3] (11.6, 7.5) rectangle (13, 8.1);
		\draw (14.0, 7.75) node{$\mathcal{R}_1\cup \mathcal{R}_1'$};
		\draw [color=mypink1,fill=mypink1] (11.6, 6.5) rectangle (13, 7.1);
		\draw (15.9, 6.75) node{previously known results};
		%    \node (node001) at (11.6, 8)  [draw,fill=red]{\qquad};
		
		\draw (15.8, 4.5) node{$P_*:=\begin{cases}(\frac{3(n-1)}{2(3n+1)},\frac{3(n-1)}{2(3n+1)}),\quad n \;\;\text{odd}\\
				(\frac{3n-2}{2(3n+2)},\frac{3n-2}{2(3n+2)}),\quad n\;\;\text{even}\end{cases}$};
		\draw (16.6, 2.5) node{$P_0:=\begin{cases}(\frac{(n+5)(n-1)}{2(n^2+4n-1)},\frac{(n-1)(n+3)}{2(n^2+4n-1)}),\quad n \;\;\text{odd}\\
				(\frac{n^2+3n-6}{2(n^2+3n-2)},\frac{(n-1)(n+2)}{2(n^2+3n-2)}),\quad n\;\;\text{even}\end{cases}$};
	\end{tikzpicture}\caption{}\label{figthm}
\end{figure}
\begin{theorem}\label{th-resolvent}
	Let $T^\lambda$ be a H\"ormander type operator with the convex and straight conditions.  If $p>\frac{2n}{n-1}$ and 
	\beq\label{eq:3a}
	\|T^\lambda f\|_{L^p(\R^n)}\lesssim \|f\|_{L^p(B^{n-1}_1(0))},
	\eeq
	then
	\begin{equation}\label{eq:aaa}	
		\|(-\Delta-z)^{-1}\|_{p\to p}\lesssim_{p,n} |z|^{-1+\gamma_{p,p}}  \dist(z,[0,\infty))^{-\gamma_{p,p}}.
	\end{equation}
\end{theorem}
\begin{remark}
	Indeed, the proof of \eqref{eq:aaa} can be reduced to showing a multiplier estimate 
	\beqq
	\Big\|\mathcal{F}^{-1}\Big(\frac{\tilde{\chi}(\xi)\hat{f}(\xi)}{|\xi|^2-1-i\delta}\Big)\Big\|_p\lesssim |\delta|^{-\gamma_{p,p}}\|f\|_p
	\eeqq
	where$\tilde{\chi}\in C_0^\infty(1-2\delta_0,1+2\delta_0)$ for a small $\delta_0>0$ and $0<|\delta|\ll 1$. Using the Carleson-Sj\"olin reduction as displayed in subsection \ref{sub1}, an important ingredient in the approach is the following oscillatory integral operator estimate 
	\beqq
	\big\|\int_{\R^{n-1}} e^{2\pi i \Psi(x,y')}a(x,y')f(y')d y'\big\|_{L^p(\R^n)}\lesssim \|f\|_{L^p},
	\eeqq
	where $\Psi$ is defined as in \eqref{phase1} and $a\in C_c^\infty(B_1^{n}(0)\times B_1^{n-1}(0))$.  Since $\Psi(x,y')$, up to a diffemorphism in $x$, satisfies the conditions ${\rm H}_1,{\rm H}_2,{\rm H}_3,{\rm H}_4$, we may apply  Conjecture \ref{coj} to get the desired results.
\end{remark}

As a direct consequence of Theorem \ref{th-resolvent} and \eqref{eq:003}, by interpolation and the epsilon removal arguments, up to a pair of  intervals $(B,D)\cup (B',D')$, we obtain Conjecture \ref{coj} implies Conjecture \ref{conjecture}. Furthermore, we may  use the new oscillatory integral estimates in Theorem \ref{theo} to further improve the range of $p$ in \eqref{eq:aaa}.

\section{Reductions}\label{reduc}
Typically speaking,  the phase   $\phi(x,\xi)$ which satisfies  the conditions ${\rm H}_1,{\rm H}_2, {\rm H}_3,{\rm H}_4$ can be viewed as a small perturbation of the translation invariant  case. More precisely, through changing of variables,  it can  be rewritten as 
\beq\label{eq:09}
\phi(x,\xi)=\langle x',\xi\rangle+x_nh(\xi)+\mathcal{ E}(x,\xi),
\eeq 
where $h$ and $\mathcal{E}$ are smooth functions, $h$ is quadratic in $\xi$ and $\mathcal{ E}$ is quadratic in $x,\xi$. However, under the new coordinate, the formula of $\phi$ in \eqref{eq:09} may not satisfy the straight condition, even though ${\rm H}_1,{\rm H}_2,{\rm H}_3$ can  be ensured. In other words, the straight condition may not be kept under a general  diffeomorphism  in the spatial variables. Therefore, we should be careful when performing the change of variables in $x$ and,meanwhile, keeping  track  of the straight condition.

{\bf \noindent Basic reductions}
As mentioned above, the straight condition may be destroyed while performing a  diffeomorphism with respect to the spatial variables, which inspires us  to consider a wider class of functions which, upon a diffeomorphism in the spatial variables, satisfy  the straight condition. To formalize that, we introduce a notion of ${\bf \Phi}_{\rm cs}$. 
\begin{definition}
We say a function $\phi(x,\xi)$ lies in the class ${\bf \Phi}_{\rm cs}$,  if, modulo a diffeomorphism in the spatial variables $x$,  $\phi(x,\xi)$ satisfies the conditions ${\rm H}_1,{\rm H}_2,{\rm H}_3,{\rm H}_4$.
\end{definition} 
\begin{remark}
In terms of ${\bf \Phi}_{\rm cs}$, it is an interesting problem to investigate the influence of  the higher order terms of $\phi(x,\xi)$ in $x$.  	
\end{remark}
\begin{example}
In \cite{Bourgain1}, Bourgain disproved Hormander's conjecture by constructing a counterexample where the Kakeya compression phenomena happen which roughly say  that  the main contribution to  the oscillatory integral may be concentrated  in a lower dimensional submanifold.  Next we will analyse Bourgain's counterexample to vividly show that there does not exist a diffeomorphism  in the spatial variables such that the Gauss map $G(x,\xi)$ is invariant when $x$ changes.

Let $$P(x,y)=x_1y_1+x_2y_2+2x_3y_1y_2+x_3^2y_1^2.$$
Assume that there exists a diffeomorphism 
\beqq
x\rightarrow \kappa(\tilde x), 
\eeqq such that  for given $\xi$, the associated $G(x,\xi)$ keeps invariant when $x$ changes,
where $x=(x_1,x_2,x_3)=(\kappa_1(\tilde x),\kappa_2(\tilde x),\kappa_3(\tilde x))$.
Then the tangent space of the hypersurface $\{\partial_{\tilde x}P(\tilde x,y): y\in B_{1}^{n-1}(0)\}$ at point $(\tilde x,y)$ can be spanned by the following two linear independent 
vectors 
\beqq
\partial_{xy}P(x,y)|_{x=\kappa(\tilde x)}\big(\frac{\partial \kappa}{\partial \tilde x}\big),
\eeqq
where
\begin{equation}
	\partial_{xy}P(x,y)=\left( \begin{array}{ccc}
		1 & 0&2y_2+4x_3y_1\\
		0 & 1 & 2y_1
	\end{array}\right) .
\end{equation}
%Therefore it can also be spanned by 
%\begin{equation}
%	\partial_{xy}\Psi(x,y)=\left( \begin{array}{ccc}
%		1 & -2x_3&2y_2\\
%		0 & 1 & 2y_1
%	\end{array}\right) .
%\end{equation}
%Therefore there does not exist a diffemorphis $\kappa$ such that 
%$$\big(\frac{\partial \kappa}{\partial \tilde x}\big)\left( %\begin{array}{ccc}
%	1 & -2x_3&2y_2\\
%	0 & 1 & 2y_1
%\end{array}\right)=\left( \begin{array}{ccc}
%1 & 0&2y_2\\
%0 & 1 & 2y_1
%\end{array}\right).$$
We claim that there does not exist a diffeomorphism $\kappa$ such that
\begin{equation} \label{eq}
	\left( \begin{array}{ccc}
		1 & 0&2y_2+4\kappa_3(\tilde x)y_1\\
		0 & 1 & 2y_1
	\end{array}\right)\big(\frac{\partial \kappa}{\partial \tilde x}\big)=\left( \begin{array}{ccc}
		1 & 0&C_1(y_1,y_2)\\
		0 & 1 & C_2(y_1,y_2)
	\end{array}\right),
\end{equation}
where $C_1(y_1,y_2),C_2(y_1,y_2)$  only depend  on $y_1,y_2$. Indeed, if  \eqref{eq} holds, we have
\begin{align*}
	& \frac{\partial\kappa_1(\tilde x)}{\partial\tilde{x}_1}+\left(2y_2+4\kappa_3(\tilde x)y_1\right)\frac{\partial\kappa_3(\tilde x)}{\partial\tilde{x}_1}=1,\\
	& \frac{\partial\kappa_1(\tilde x)}{\partial\tilde{x}_2}+\left(2y_2+4\kappa_3(\tilde x)y_1\right)\frac{\partial\kappa_3(\tilde x)}{\partial\tilde{x}_2}=0,\\
	&\frac{\partial\kappa_1(\tilde x)}{\partial\tilde{x}_3}+\left(2y_2+4\kappa_3(\tilde x)y_1\right)\frac{\partial\kappa_3(\tilde x)}{\partial\tilde{x}_3}=C_1(y_1,y_2).
\end{align*}
By solving the  equations, one has $\kappa_1(\tilde x)=-2y_2\kappa_3(\tilde x)-2y_1\kappa_3(\tilde x)^2+\tilde{x}_1+C_1(y_1,y_2)\tilde{x}_3$.
Since $C_1(y_1,y_2),C_2(y_1,y_2)$ are constants when $y_1,y_2$ are fixed, we get 
$$\kappa_1(\tilde x)=-2\kappa_3(\tilde x)-2\kappa_3(\tilde x)^2+\tilde{x}_1+C_1(1,1)\tilde{x}_3=-\kappa_3(\tilde x)-\kappa_3(\tilde x)^2+\tilde{x}_1+C_1(1/2,1/2)\tilde{x}_3,$$
then $\kappa_3(\tilde x)+\kappa_3(\tilde x)^2=c_1\tilde{x}_3$, where $c_1=C_1(1,1)-C_1(1/2,1/2)$.
By the same argument, we have $\kappa_2(\tilde x)=-2y_1\kappa_3(\tilde x)+\tilde{x}_2+C_2(y_1,y_2)\tilde{x}_3$ and $\kappa_3(\tilde x)=c_2\tilde{x}_3$. Thus $c_2\tilde{x}_3+c_2^2\tilde{x}_3^2=c_1\tilde{x}_3$ holds for all $\tilde{x}_3\in\mathbb{R}$, which implies 
$c_1=c_2=0$ and  $\kappa_3(\tilde x)=0$, this is a contradiction since we assume  $\kappa$ is a diffemorphism.
\end{example}

In addition, we also assume some additional  quantitative conditions on $\phi$. Firstly, let's introduce a notion of \emph{reduced form}.
\begin{definition}
We say a function $\phi(x,\xi)$ is of  reduced form if $\phi \in{\bf \Phi}_{\rm cs}$ with the following conditions hold: let $\varepsilon>0$ be a fixed constant and $a(x,\xi)$ be  supported on $X\times \Omega$, where $X:=X'\times X_n$ and $X'\subset B^{n-1}_1(0), X_n\subset (-1,1)$ and $\Omega \subset B^{n-1}_1(0)$, upon which the phase $\phi$ has the form
\beqq
\phi(x,\xi)=\langle x',\xi\rangle+x_n h(\xi)+\mathcal{E}(x,\xi),
\eeqq
with 
\beq\label{eq:uni}
|\partial_x^\alpha \partial_\xi^\beta \phi(x,\xi)|\leq  C_{\alpha,\beta},\quad |\alpha|,|\beta|\leq N_{\rm par},
\eeq
here  $h$ and $\mathcal{E}$ are smooth functions and $h$ is quadratic in $\xi$, $\mathcal{E}$ is quadratic in $x,\xi$ and $N_{\rm par}$ is a given large constant.

Furthermore, $\phi$ also satisfies  the following conditions:
\begin{itemize}
	\item[${\bf C_1}:$] The eigenvalues of the Hessian $\big(\tfrac{\partial^2 h}{\partial x_i \partial x_j}\big)_{(n-1)\times(n-1)}$ all fall into $[1/2,2]$.
	\item[${\bf C_2}:$] Let  $c_{\rm par}>0$ be a small constant, $N_{\rm par}>0$ be a given large constant as above, 
	\beqq
	|\partial_x^\alpha \partial_\xi^\beta \mathcal{E}(x,\xi)|\leq c_{\rm par},\quad |\alpha|,|\beta|\leq N_{\rm par}.
	\eeqq
	\end{itemize}
	\end{definition}
Let $1\leq R\leq \lambda, T^\lambda$ be defined with the reduced form and $Q_p(\lambda, R)$ be the optimal constant such that 
\beq\label{eq:op}
\|T^\lambda f\|_{L^p(B_R^n(0))}\leq Q_p(\lambda,R)\|f\|_{L^2}^{\frac{2}{p}}\|f\|_{L^\infty}^{1-\frac{2}{p}}.
\eeq
We claim that the proof of Theorem \ref{theo} can be reduced to showing that for $p>p_n$ and for each $\varepsilon>0$,
\beq\label{eq:red}
Q_p(\lambda, R)\lesssim_{\varepsilon,p}R^{\varepsilon}.
\eeq
Indeed, we  firstly claim that:
 
{\bf Claim:} If $\mathcal{T}^\lambda$ is an operator satisfying  the conditions ${\rm H}_1,{\rm H}_2,{\rm H}_3,{\rm H}_4$, then 
\beq\label{claim1}
\|\mathcal{T}^\lambda f\|_{L^p(B_R^n(0))}\lesssim_{\phi}\|T^{\lambda\tilde{r}^2} \tilde{f}\|_{L^p(B_{CR}^n(0))},
\eeq
where $T^\lambda$ is defined with the reduced form, $\tilde r>0$ is an appropriate constant depending on $\phi$ and 
\beq\label{claim2}
\|\tilde{f}\|_{L^p}\lesssim_\phi  \|f\|_{L^p}.
\eeq
 We take the above claim for granted and prove  \eqref{eq:red} implies Theorem \ref{theo}.  To be more precise, we need to show  \eqref{eq:red} implies
\beqq
\|\mathcal{T}^\lambda f\|_{L^p(B_R^n(0))}\lesssim_{\varepsilon, p,\phi}R^\varepsilon \|f\|_{L^p}.
\eeqq
Indeed, by \eqref{eq:red},\eqref{claim1}, \eqref{claim2}, we have 
\beq
\|\mathcal{T}^\lambda f\|_{L^p(B_R^n(0))}\lesssim_{\varepsilon,p,\phi}R^\varepsilon \|f\|_{L^2}^{\frac{2}{p}}\|f\|_{L^\infty}^{1-\frac{2}{p}}.
\eeq
By taking  $f=\chi_E$,  we get
\beqq
\|\mathcal{T}^\lambda f\|_{L^p(B_R^n(0))}\lesssim_{\varepsilon,p,\phi}R^\varepsilon\|f\|_{L^p}.
\eeqq
Then the desired results follows by interpolation argument. Therefore, it suffices to verify the claim.
For convenience, we just need to track the phase when changing of variables.

The proof can be obtained by modifying the associated part in \cite{GHI}.  Without loss of generality, we may assume 
\beqq
\partial_x^\alpha \phi(x,0)=0,\quad \partial_\xi^\alpha \phi(0,\xi)=0, \quad \alpha\in \mathbb{Z}^n.
\eeqq
Otherwise, we take $\phi$ to be 
\beqq
\phi(x,\xi)+\phi(0,0)-\phi (0,\xi)-\phi(x,0).
\eeqq
By Taylor's formula, we have 
\beqq
\phi(x,\xi)=\partial_\xi \phi(x,0)\cdot \xi +\rho(x,\xi),
\eeqq
where $\rho(x,\xi)$ is quadratic in $\xi$.
By the condition ${\rm H}_1$,  we may assume ${\rm rank}\partial_{x'\xi}\phi=n-1$ and $G(0,0)=(0,\cdots,1)$, thus we may find a smooth function $\Phi(x',x_n,0)$ such that 
\beqq
\partial_\xi \phi(\Phi(x',x_n,0),x_n,0)=x'.
\eeqq
By our assumption, one may also get 
\beq\label{eq:1231}
\Phi(0,0)=0,\quad \partial_{x_n}\Phi(0,0)=0, \quad \partial_{x'}\Phi(0,0)=\partial_{x'\xi}^2\phi(0,0)^{-1}.
\eeq
By changing of variables 
\beqq
x'\longrightarrow \Phi(x',x_n,0), \quad x_n\longrightarrow x_n,
\eeqq
thus it suffices to consider
\beqq
\langle x',\xi \rangle+\rho(\Phi(x',x_n,0),x_n,\xi).
\eeqq
Then taking another expansion in $x$  and using \eqref{eq:1231} yield
\beq
\begin{aligned}
\rho(\Phi(x',x_n,0),x_n,\xi))&=\rho(\Phi(0,0),0,\xi)+\partial_{x'}\rho(0,\xi)\partial_{x'}\Phi(0)x'\\&+\big(\partial_{x_n}\rho\big)(0,\xi)x_n+O(|x|^2|\xi|^2).
\end{aligned}
\eeq
Finally, from \eqref{eq:1231}, one deduces that 
\beqq
\phi(x,\xi)=\langle x',\xi+\partial_{x'\xi}\phi(0,0)^{-T} \partial_{x'}\rho(0,\xi)\rangle+x_n \partial_{x_n} \rho (0,\xi)+O(|x|^2|\xi|^2).
\eeqq
Then by changing of variables
\beqq
\xi+\partial_{x'\xi}\phi(0,0)^{-T} \partial_{x'}\rho(0,\xi)\longrightarrow \xi,
\eeqq
and taking $h(\xi)=\partial_{x_n}\rho(0,\xi)$, we have 
\beqq
\phi(x,\xi)=\langle x',\xi \rangle+x_n h(\xi)+O(|x|^2 |\xi|^2).
\eeqq
Since $\Omega \subset B_1^{n-1}(0)$, we partition $\Omega$ into a  family of balls $\{B_\alpha\}$ of radius $r$ and center $\xi_\alpha$, such that 
\beqq
\Omega \subset \bigcup_\alpha B_\alpha.
\eeqq
By triangle inequality, it suffices to consider a single ball $B_\alpha$. By changing of variables 
\beqq
\xi \longrightarrow \tilde r\xi+\xi_\alpha,
\eeqq
where $\tilde r \geq r$,  under the new coordinates, we just need to consider $\xi \in B_{r/\tilde{r}}^{n-1}(0)$
and 
\beqq
e^{2\pi i\phi^\lambda(x,\xi_\alpha)}\int e^{2\pi i (\phi^\lambda(x,\xi)-\phi^\lambda(x,\xi_\alpha))}a^\lambda(x,\xi)f(\xi)d\xi.
\eeqq
Since $\phi(x,\xi)=\langle x',\xi \rangle+x_n h(\xi)+\mathcal{E}(x,\xi)$, we have 
\beqq
\phi^\lambda(x,\xi)-\phi^\lambda(x,\xi_\alpha)=\tilde{r} \partial_\xi\phi^\lambda(x,\xi_\alpha)\cdot \xi+\tilde{r}^2x_n\tilde h(\xi)+\tilde{r}^2\widetilde{\mathcal{E}}^\lambda(x,\xi),
\eeqq
where
\beq\label{change}
\begin{aligned}
\tilde{h}(\xi):&=\tilde{r}^{-2}(h(\tilde{r}\xi+\xi_\alpha)-h(\xi_\alpha)-\tilde{r}\partial_{\xi}h(\xi_\alpha)\cdot \xi)\\
\widetilde{\mathcal{E}}^\lambda(x,\xi):&=\tilde{r}^{-2}(\mathcal{E}^\lambda(x,\tilde{r}\xi+\xi_\alpha)-\mathcal{E}^\lambda(x,\xi_\alpha)-\tilde{r}\partial_{\xi}\mathcal{E}^\lambda(x,\xi_\alpha)\cdot \xi).
\end{aligned}
\eeq
By another change of variables in $x$ as follows
\beqq
x' \longrightarrow  \lambda\Phi\Big( \frac{x'}{\lambda\tilde{r}},\frac{x_n}{\lambda\tilde{r}^2},\xi_\alpha\Big), \quad  x_n\longrightarrow \tilde{r}^{-2} x_n,
\eeqq
finally, it suffices to consider
\beq\label{eq:213}
\tilde{\phi}:=\langle x',\xi\rangle+x_n \tilde{h}(\xi)+\bar{\mathcal{E}}^{\lambda \tilde{r}^2}(x,\xi),
\eeq
where $\bar{\mathcal{E}}(x,\xi):=\widetilde{\mathcal{E}}(\Phi(\tilde{r}x',x_n,\xi_\alpha),x_n,\xi)$.

From \eqref{change}, $\tilde{h}$ is quadratic in $\xi$ and $\bar{\mathcal{E}}(x,\xi)$ is quadratic in $x,\xi$. Furthermore, through an affine change of variables in $\xi$ and by choosing appropriate small $\tilde{r}$ such that $r/\tilde{r}$ is also sufficiently small, we may ensure the condition ${\rm C}_1$ and ${\rm C}_2$. Define $T^\lambda$  with the phase in \eqref{eq:213} and note that all the implicit constants arising when performing the change of variables depend on $\phi$, we will obtain \eqref{claim1} and \eqref{claim2}.
\subsection{Further remarks on the phase}
Let  $\phi$ be of the reduced form, by our assumption,  we may choose a smooth function $p(x)$ such that  $\phi(p(x),\xi)$ satisfies the  straight condition and
\beq\label{eq:uni1}
|\partial^\alpha_x \partial_\xi ^\beta \phi(p(x),\xi)|\leq \bar{C}_{\alpha,\beta},\quad |\alpha|, \beta|\leq N_{\rm par},
\eeq
uniformly.
Indeed, we may always choose a function $p:\R^{n-1}\rightarrow \R^{n-1}$  with $
|\partial_x p(x)|\lesssim1 $
such that $\phi(Ap(\frac{x}{A}),\xi)$ satisfies the straight condition where $A$ is large enough to ensure \eqref{eq:uni1}.

\section{wave packet decomposition}\label{wav} 

Let $r \geq 1$ and $\Theta_r$ be a collection of cubes  $\{\theta\}$ of sidelength $\frac{9}{11}r^{-1/2}$  and center $\xi_\theta$ which cover the ball $B_2^{n-1}(0)$. Correspondingly, we take a smooth partition of unity $\{\psi_\theta\}_{\theta\in \Theta_r}$ with respect to the cover $\Theta_r$. Let $\tilde \psi_\theta$ be a non-negative smooth cut-off function supported on $\frac{11}{9}\theta$ and equal to $1$ on $\frac{11}{10}\theta$.  Given a function $g$, by taking Fourier series expansion, we have 
\beqq
g(\xi)\psi_\theta(\xi)\cdot \tilde{\psi}_\theta(\xi)=\Big(\frac{r^{1/2}}{2\pi}\Big)^{n-1}\sum_{v\in r^{1/2}\mathbb{Z}^{n-1}}(g\psi_\theta)^{\widehat{}}(v)e^{2\pi i v\cdot \xi}\tilde{\psi}_\theta(\xi).
\eeqq
Define 
\beqq
g_{\theta,v}(\xi):=\Big(\frac{r^{1/2}}{2\pi}\Big)^{n-1}(g\psi_\theta)^{\widehat{}}(v)e^{2\pi i v\cdot \xi}\tilde{\psi}_\theta(\xi).
\eeqq
Correspondingly, we may make the following decomposition
\beqq
g=\sum_{(\theta,v)\in \Theta_r\times r^{1/2}\mathbb{Z}^{n-1}}g_{\theta,v}.
\eeqq
Let $1\leq r\leq R$ and $B_{r}^n(x_0)\subset B_R^n(0)$, define 
\beqq
\phi^\lambda_{x_0}(x,\xi):=\phi^\lambda(x,\xi)-\phi^\lambda(x_0,\xi).
\eeqq
By the assumption of the phase, there exists $\gamma_{\theta,v,x_0}^\lambda(x_n)$ such that 
\beqq
\partial_\xi\phi^\lambda_{x_0}(\gamma_{\theta,v,x_0}^\lambda(x_n),x_n,\xi_\theta)+v=0.
\eeqq
Given $\theta, v$, define a tube $T_{\theta,v}=T_{\theta,v}(x_0)$ to be 
\beq\label{wave-1}
T_{\theta,v}(x_0):=\{(x',x_n):|x'-\gamma_{\theta,v,x_0}^\lambda(x_n)|\lesssim r^{\frac{1+\delta}{2}}, |x_n-x_0^n|\leq Cr\},
\eeq
and
\beqq
g_{T_{\theta,v}}:=e^{-2\pi i \phi^\lambda(x_0,\xi)}(g(\cdot) e^{2\pi i\phi^\lambda(x_0,\cdot)})_{\theta,v}.
\eeqq
Thus, we have 
\beqq
T^\lambda g(x)=\sum_{\theta, v}T^\lambda g_{T_{\theta,v}}(x).
\eeqq
We define a collection of tubes associated to the ball $B_r^n(x_0)$ by 
\beqq
\mathbb{T}[B_r^n(x_0)]:=\{T_{\theta,v}(x_0):(\theta,v)\in \Theta_r\times r^{1/2}\mathbb{Z}^{n-1}\}.
\eeqq
The main contribution of $T^\lambda g_{T_{\theta,v}}$ is concentrated on $T_{\theta,v}$ and rapidly decays outside of the tube which can be manifested  in the following lemma.
\begin{lemma}
If $x \in B_r^n(x_0)\backslash T_{\theta,v}$, then 
\beqq
|T^\lambda g_{T_{\theta,v}}(x)|\lesssim_N (1+r^{-1/2}|\nabla_\xi\phi_{x_0}^\lambda(x,\xi_\theta)+v|)^{-N}{\rm RapDec}(r)\|g\|_{L^2}.
\eeqq
\end{lemma}
\begin{proof}
For convenience, we use $T$ to denote $T_{\theta,v}$ and use $g_{x_0}$ to denote $ge^{2\pi i \phi^\lambda(x_0,\xi)}$. Recall the definition of $g_{T}$, we have 
\beqq
T^\lambda g_{T}(x)=\Big(\frac{r^{\frac{1}{2}}}{2\pi}\Big)^{n-1}(g_{x_0}\psi_\theta)^{\widehat{}}(v)\int e^{2\pi i \phi^\lambda(x,\xi)-2\pi i \phi^\lambda(x_0,\xi)} e^{2\pi i v\cdot \xi} a^\lambda(x,\xi)\widetilde{\psi}_\theta(\xi)d\xi.
\eeqq
By changing of variables: $\xi\longrightarrow r^{-1/2}\xi+\xi_\theta$, it suffices to consider the integral
\beqq
\int e^{2\pi i \phi^\lambda(x,r^{-1/2}\xi+\xi_\theta)-2\pi i \phi^\lambda(x_0,r^{-1/2}\xi+\xi_\theta)} e^{2\pi i r^{-1/2} v\cdot \xi} a^\lambda(x,r^{-1/2}\xi+\xi_\theta)\widetilde{\psi}(\xi)d\xi.
\eeqq
Taking the derivative in $\xi$, we get	
\beqq
\begin{aligned}
&\quad \partial_{\xi} ( \phi^\lambda(x,r^{-1/2}\xi+\xi_\theta)- \phi^\lambda(x_0,r^{-1/2}\xi+\xi_\theta)+r^{-1/2}v\cdot \xi)\\
&=r^{-1/2}(\partial_\xi\phi^\lambda(x,r^{-1/2}\xi+\xi_\theta))-\partial_\xi\phi^\lambda(x,\xi_\theta)-(\partial_\xi\phi^\lambda(x_0,r^{-1/2}\xi+\xi_\theta)-\partial_\xi\phi^\lambda(x_0,\xi_\theta))\\
&\quad \quad \quad +r^{-1/2}(v+\partial_\xi\phi^\lambda_{x_0}(x,\xi_\theta))\\
&=r^{-1/2}(v+\partial_\xi \phi^\lambda_{x_0}(x,\xi_\theta))+O(1).
\end{aligned}
\eeqq
Integration by parts, 	we will obtain the desired results.

\end{proof}
We also have the following $L^2$-orthogonality properties.
\begin{lemma}[$L^2$-orthogonality]\label{orth}
	For any $\mathbb{T}\subset \mathbb{T}[B_r^n(x_0)]$, it holds that 
	\beq
	\|\sum_{T\in \mathbb{T}}g_T\|_2^2\lesssim \sum_{T\in \mathbb{T}}\|g_T\|_2^2\lesssim \|g\|_2^2.
	\eeq 
	Moreover, if $\mathbb{T}$ is any collection of tubes with the same $\theta$, then 
	\beqq
	\|\sum_{T\in \mathbb{T}}g_T\|_2^2\sim \sum_{T\in \mathbb{T}}\|g_T\|_2^2.
	\eeqq
	\end{lemma}

{\bf \noindent Comparing wave-packet at different scales.}
Let $r^{1/2}<\rho <r$. Consider another smaller ball $B_\rho^n(\tilde x_0)\subset B_r^n(x_0)$. Similarly, we may define the wave-packet decomposition with respect to the ball $B_{\rho}^n(\tilde x_0)$. To distinguish the wave packet of different scales, we use $\widetilde{\mathbb{T}}[B_\rho^n(\tilde{x}_0)]$ to denote the  smaller scale wave-packet.
\begin{definition}
	We say a function $h$ is concentrated on wave packets from a tube set $\mathbb{T}_\alpha$, if 
	\beq\label{eq:deh}
	h=\sum_{T\in \mathbb{T}_\alpha}h_{T}+{\rm RapDec}(r)\|h\|_2.
	\eeq
	\end{definition}
\begin{definition}
	Let $(\theta,v)\in \Theta_r \times r^{1/2}\mathbb{Z}^{n-1}$ and let $(\tilde \theta,\tilde v)\in \Theta_\rho \times \rho^{1/2}\mathbb{Z}^{n-1}$. We define a  set $\widetilde{\mathbb{T}}_{\theta, v}[B_\rho^n(\tilde x_0)]$ collection of smaller tubes as follows 
	\beqq
	\widetilde{\mathbb{T}}_{\theta,v}[B_\rho^{n}(\tilde x_0)]:=\{\tilde T_{\tilde \theta,\tilde v}\in \widetilde{\mathbb{T}}[B_\rho^n(\tilde x_0)]:{\rm dist}(\theta,\tilde \theta)\lesssim \rho^{-1/2}, |\tilde v-(\partial_\xi\phi^\lambda_{x_0}(\tilde x_0,\xi_\theta)+v)|\lesssim r^{(1+\delta)/2}\}.
	\eeqq
	\end{definition}
One may carry over the approach verbatim in \cite{GOW}  to obtain the following two lemmas.
\begin{lemma}[\cite{GOW}]
	Let $T_{\theta,v}\in \mathbb{T}[B_r^n(x_0)]$. Then it holds that 
	\beqq
	g_{T_{\theta,v}}=(g_{T_{\theta,v}})|_{\widetilde{\mathbb{T}}_{\theta,v}[B_\rho^n(\tilde{x}_0)]}+{\rm RapDec}(r)\|g\|_2.
	\eeqq
	\end{lemma}

\begin{lemma}[\cite{GOW}]
	Assume $T_{\theta,v}\subset \mathbb{T}[B_r^n(x_0)]$. If $\tilde T_{\tilde \theta,\tilde v}\in \widetilde{\mathbb{T}}_{\theta,v}[B_\rho^n(\tilde{x}_0)]$, then it holds that 
	\beqq
	{\rm HausDist}(\tilde T_{\tilde \theta,\tilde v},T_{\theta,v}\cap B_\rho^n(\tilde x_0))\lesssim r^{(1+\delta)/2},
	\eeqq
	and 
	\beqq
	\measuredangle(G(\xi_\theta),G(\xi_{\tilde{\theta}}))\lesssim  \rho^{-1/2}.
	\eeqq
	\end{lemma}

\section{Transverse equidistribution property}\label{tran}
Transverse equidistribution property is based on a simple observation  which can be roughly stated as follows: if ${\rm supp}\hat{g}\subset B_r^n(0), r>0$, then $g$ can not be concentrated in a ball of  radius less than $r^{-1}$. Starting from this fact and other geometric assumptions, Guth \cite{Guth} established the transverse equidistribution lemma for the extension operator. Then Guth-Hickman-Iliopoulou \cite{GHI} extended  it to the H\"ormander type operator with the convex condition. It should be noted that the proof of the transverse equidistribution lemma in \cite{GHI} relies on the phase belongs to a category  which may not satisfy the straight condition. To overcome this obstacle,  we follow the approach  in \cite{GOW} which deals with the input function directly without recourse to  a further operation under $T^\lambda$. Since $T^\lambda$ satisfies the straight condition, i.e., for given $\xi$,  $G^\lambda(x,\xi)$ keeps invariant when $x$ changes. Thus, we may use $G(\xi)$ to denote $G^\lambda(x,\xi)$. It is worth noting that there are  still some differences between \cite{GOW} and our case at this point, for example, in \cite{GOW}, 
$$G(\xi)=\frac{(-\xi,1)}{\sqrt{1+|\xi|^2}},$$
therefore, if $V\subset \R^n$ is a subspace, then the set 
$$\mathcal{S}:=\{\xi\in \R^n: G(\xi)\in V\}$$
falls into  an affine subspace in $\R^n$. However, in our case, the associated set  $\mathcal{S}$ may be  a curved submanifold which requires  more  technical handling.

Since $\phi(x,\xi)$ satisfies the straight condition, in terms of the Gauss map,  it suffices to consider the following class of varying hypersurfaces  $\{\partial_x\phi(x,\xi):\xi \in \Omega\}$ at $x=0$. By the  condition ${\rm H}_1$, we may find locally  a function $q:\R^{n-1}\longrightarrow \R^{n-1}$ such that 
\beqq
\partial_{x'}\phi(0,q(\xi))=\xi.
\eeqq
Define 
\beq\label{eq:hxi}
h(\xi):=\partial_{x_n}\phi(0,q(\xi)).
\eeq
We may assume the hypersurface $\{\partial_x\phi(0,\xi):\xi \in \Omega\}$ can be reparameterized by $\{(\xi,h(\xi): \xi \in \Omega)\}$ with 
\beq\label{eq:hass}
\|\partial_{\xi\xi}^2 h(\xi)-I_{n-1}\|_{\rm op}\ll 1, \quad \xi \in \Omega.
\eeq
Otherwise,  we can  choose a non-degenerate matrix $A$ such that 
\beqq
\partial_{\xi\xi}^2 h(A\xi )|_{\xi=0}=I_{n-1}.
\eeqq
By choosing the support of $\xi$ sufficiently small, it holds that 
\beqq
\big\|\partial_{\xi\xi}^2h(A \xi)-I_{n-1}\big\|_{\rm op}\ll1.
\eeqq
Correspondingly, we make  another affine transformation  in $x$ and  replace $\phi(x,\xi)$ by $\tilde{\phi}(x,\xi)$ which is defined by 
\beqq
\tilde \phi(x,\xi):=\phi(A^{-1}x',x_n,q(A\xi)).
\eeqq
Obviously, $\tilde\phi(x,\xi)$ satisfies the straight conditions and 
\beqq
\partial_{x}\tilde\phi(x,\xi)|_{x=0}=(\xi, \partial_{x_n}\phi(0,q(A\xi))).
\eeqq
\begin{definition}
	Let $P_1,\cdots,P_{n-m}:\mathbb{R}^n\rightarrow \mathbb{R}$ be polynomials. We consider the common zero set 
	\beq\label{al1}
	Z(P_1,\cdots,P_{n-m}):=\{x\in \R^n: P_1(x)=\cdots=P_{n-m}(x)=0\}.
	\eeq
	Suppose that for all $z\in Z(P_1,\cdots,P_{n-m})$, one has 
	\beqq
	\bigwedge_{j=1}^{n-m}\nabla P_j(x)\neq 0.
	\eeqq
	Then a connected branch of this set, or a union of connected branches of this set, is called an $m$-dimensional transverse complete intersection. Given a set $Z$ of the form \eqref{al1}, the degree of $Z$ is defined by  
	\beqq
	\min\big(\prod_{j=1}^{n-m}{\rm deg}(P_i)\big),
	\eeqq
	where the minimum is taken over all possible representations of $Z=Z(P_1,\cdots,P_{n-m})$.
\end{definition}
\begin{definition}
	Let $r\geq 1$ and $Z$ be an $m$-dimensional transverse complete intersection. A tube $T_{\theta,v}(x_0)\in \mathbb{T}[B_{r}^n(x_0)]$ is said to be $r^{-1/2+\delta_m}$-tangent to $Z$ in $B_r^n(x_0)$ if it satisfies 
\begin{itemize}
	\item $T_{\theta,v}(x_0)\subset N_{r^{1/2+\delta_m}}(Z)\cap B_r^n(x_0);$
	\item For every $z\in Z\cap B_r^n(x_0)$, if there is $y\in T_{\theta,v}(x_0)$ with $|z-y|\lesssim r^{1/2+\delta_m}$, then one has 
	\beqq
	\measuredangle(G(\theta),T_zZ)\lesssim r^{-1/2+\delta_m}.
	\eeqq
	Here, $T_zZ$ is the tangent space of $Z$ at $z$ and 
	\beqq
	G(\theta):=\{G(\xi):\xi \in \theta\}.
	\eeqq
	\end{itemize}
\end{definition}
\begin{definition}
	Let $1\leq \rho \leq r$ and $Z$ be an $m$-dimensional transverse complete intersection and let $B_{\rho}^n(\tilde x_0)\subset B_r^n(x_0)$. Define a collection of tangent tubes inside  a ball  as 
	\beqq
	\mathbb{T}_Z[B_r^n(x_0)]:=\{T\in \mathbb{T}[B_r^n(x_0)]: T \;\text{is}\; r^{-1/2+\delta_m}\text{-tangent to}\; Z \; \text{in}\; B_r^n(x_0)\}.
	\eeqq
	Given an arbitrary translation $b\in \R^n$, define 
	\beqq
	\widetilde{\mathbb{T}}_b[B_{\rho}^n(\tilde x_0)]:=\{\tilde T\in \widetilde{\mathbb{T}}[B_\rho^n(\tilde x_0)]: \tilde T \;\text{is}\; \rho^{-1/2+\delta_m}\; \text{-tangent to}\; Z+b \; \text{in}\; B_\rho^n(\tilde x_0)\}.
	\eeqq
For simplicity, we also abbreviate  $T_Z$ and $\widetilde {\mathbb{T}}_b$  for $\mathbb{T}_Z[B_r^n(x_0)]$ and $\widetilde{\mathbb{T}}_b[B_\rho^n(\tilde x_0)]$ respectively. 
\end{definition}
We may state our main results in this section as follows.
\begin{lemma}\label{trans}
	Let $|b|\lesssim r^{1/2+\delta_m}$. Suppose that $h$ is concentrated on large wave packets from $\mathbb{T}_Z\cap \mathbb{T}_{\tilde \theta,w}$ for some $(\tilde \theta,w)\in \Theta_\rho \times r^{1/2}\mathbb{Z}^{n-1}$. Then for every $\widetilde{W}\subset \widetilde{\mathbb{T}}_b$, we have 
	\beq\label{eq:1110}
	\big\|h|_{\widetilde{W}}\big\|_2^2\lesssim
	r^{O(\delta_m)}(r/\rho)^{-\frac{n-m}{2}}\|h\|_2^2.
	\eeq
	\end{lemma}
	As a direct consequence of Lemma \ref{trans},   we may obtain the following results.
	\begin{corollary}
	Let $|b|\lesssim r^{1/2+\delta_m}$. Suppose that $h$ is concentrated on large wave packet from $\mathbb{T}_Z$. Then for every $\widetilde{W}\subset \widetilde{\mathbb{T}}_b$, we have 
	\beqq
	\big\|h|_{\widetilde{W}}\big\|_2^2\lesssim
	r^{O(\delta_m)}(r/\rho)^{-\frac{n-m}{2}}\|h\|_2^2.
	\eeqq
\end{corollary}
\begin{proof}
	We may rewrite 
	$$h|_{\widetilde{W}}=\sum_{(\tilde \theta ,w)} h_{\tilde \theta,w},$$
such  that $h_{\tilde \theta,w}$ is concentrated on wave packets from $\mathbb{T}_Z\cap \mathbb{T}_{\tilde \theta,w}$. 	 Then, we may apply \eqref{eq:1110} to each $h_{\tilde \theta,w}$  and recall the $L^2$-orthogonality in Lemma \ref{orth}.
\end{proof}
Let $\tilde{x}_0:=(\tilde{x}_0',\tilde{x}_0^n)$,  and  define $ x_\gamma:=\gamma_{\tilde \theta,\tilde v,\tilde x_0}^\lambda(\tilde x_0^n)$, $B:=B_{Cr^{1/2+\delta}}^{n-1}(x_\gamma)$. Let $Z_0$  be the intersection of $Z+b$ of the hyperplane $\{x: x_n=\tilde x_0\}$.  Up to a  harmless small perturbation\footnote{see \cite{GOW}} in the $x_n$ direction, we may assume $Z_0$ is a  transverse complete intersection in $\R^{n-1}$.  Define a smooth map $\Phi: \R^{n-1}\rightarrow \R^{n-1}$ as follows 
\beqq
\Phi(x'):=-\partial_\xi\phi(x',\tilde{x}_0^n,\xi_{\tilde \theta}).
\eeqq
\begin{proposition}\label{eq:local}
	Let $h$ be concentrated on bigger wave packets from $\mathbb{T}_{\tilde \theta,w}\cap \mathbb{T}_Z$. Then
	\beq\label{eq:235}
	\|(h|_{\widetilde{W}})^{\widehat{}}\|_2\lesssim \|\hat{h}\chi_{N_{C\rho^{1/2+\delta_m}}(\Phi(Z_0)\cap \Phi(CB))}\|_2+{\rm RapDec}(\rho)\|h\|_2.
	\eeq
	\end{proposition}
\begin{proof}
First we claim that

	{\bf \noindent Claim:}	Let $\tilde T_{\tilde \theta,\tilde v} \in \widetilde{W}$, then 
	\beq\label{claim3}
	\Big|\sum_{\tilde T_{\tilde \theta,\tilde v}\in \widetilde{W}}\big(\rho^{\frac{n-1}{2}}e^{2\pi i \phi^\lambda(\tilde x_0,\xi)}\psi_{\tilde \theta}\big)^{\widehat{}}(\tilde v-y)\Big|\leq C \chi_{N_{C\rho^{1/2+\delta_m}}(\Phi(Z_0)\cap \Phi(CB))}(y)+{\rm RapDec}(\rho)\|h\|_2.
	\eeq
We firstly take the above claim for granted and  continue the proof of \eqref{eq:235}.
By the definition, 
\beqq
h|_{\widetilde{W}}=\sum_{\tilde T\in \widetilde W} h_{T}.
\eeqq
Therefore, by the orthogonality property, we have 
\beqq
\|(h|_{\widetilde W})\|_{2}^2\lesssim \sum_{\tilde T\in \widetilde W}\|h_{\tilde T}\|_2^2.
\eeqq
By the Plancherel's theorem, we get 
\beqq
\sum_{\tilde T\in \widetilde W}\|h_{\tilde T}\|_2^2\lesssim \rho^{n-1}\sum_{\tilde T\in \widetilde W}|(h_{\tilde x_0}\psi_{\tilde \theta})^{\widehat{}}(v)|^2\|\psi_{\tilde \theta}\|_2^2=\rho^{\frac{n-1}{2}}\sum_{\tilde T\in \widetilde W}|(h_{\tilde x_0}\psi_{\tilde \theta})^{\widehat{}}(v)|^2.
\eeqq 
Note that 
\beqq
(h_{\tilde x_0}\psi_{\tilde \theta})^{\widehat{}}(\tilde v)=\hat{h}\ast(e^{2\pi i \phi^\lambda(\tilde x_0,\cdot)}\psi_{\tilde \theta})^{\widehat{}}(\tilde v).
\eeqq
Using H\"older's inequality, we obtain 
\beq
\rho^{\frac{n-1}{2}}\sum_{\tilde T\in \widetilde W}|(h_{\tilde x_0}\psi_{\tilde \theta})^{\widehat{}}(v)|^2\lesssim \int |\hat{h}(y)|^2\Big(\sum_{\tilde T\in \widetilde{W}}|(\rho^{\frac{n-1}{2}}e^{2\pi i \phi^\lambda(\tilde x_0,\cdot)}\psi_{\tilde \theta})^{\widehat{}}(\tilde v-y)|\Big)dy.
\eeq
Then \eqref{eq:235} follows from \eqref{claim3}. Therefore, it remains to show the claim.	
	By changing of variables: $\xi\rightarrow \rho^{-1/2}\xi+\xi_{\tilde \theta}$, we have 
	\beqq
	\big(e^{2\pi i \phi^\lambda(\tilde x_0,\xi)}\psi_{\tilde \theta}\big)^{\widehat{}}(\tilde v-y)=\rho^{-\frac{n-1}{2}}\int e^{2\pi i \phi^\lambda(\tilde x_0, \rho^{-1/2}\xi+\xi_{\tilde \theta })-2\pi i (\tilde v-y)(\rho^{-1/2}\xi+\xi_{\tilde \theta})}
	\psi(\xi)d\xi.
	\eeqq
	Recall that 
	\beqq
	\partial_\xi \phi^\lambda_{x_0}(\gamma_{\tilde \theta,\tilde v,\tilde x_0}^\lambda(t),t,\xi_{\tilde \theta})=\tilde v.
	\eeqq
	A stationary phase argument shows that the above integral  is essentially nontrivial  when $y\in B_{\rho^{1/2+\delta_m}}^n(\tilde v-\partial_\xi\phi^\lambda(\tilde x_0,\xi_{\tilde \theta}))$.	
By our assumption and definition, we have 
\beqq
\gamma_{\tilde \theta,\tilde v,\tilde x_0}^\lambda(\tilde x_0^n)\subset N_{C\rho^{1/2+\delta_m}}(Z_0)\cap CB,
\eeqq	
Thus 
\beqq
\tilde v\subset N_{C\rho^{1/2+\delta_m}}(\Psi(Z_0))\cap \Psi(CB),
\eeqq	
where $\Psi:\R^{n-1}\rightarrow \R^{n-1}$ is defined by 
\beqq
\Psi(x'):=\partial_{\xi}\phi^\lambda_{x_0}(x',\tilde x_0^n,\xi_{\tilde \theta}).
\eeqq	
Recall that 
\beqq
\phi^\lambda_{\tilde{x}_0}(x,\xi)=\phi^\lambda(x,\xi)-\phi^\lambda(\tilde{x}_0,\xi),
\eeqq	
thus we obtain the desired results.		
	\end{proof}

\begin{proposition}\label{pro:2}
	Assume  $Z_0=(Z+b)\cap \{x_n=\tilde x_0^n\}$ and $B=B_{r^{1/2+\delta_m}}^n(\tilde{x}_0)$. Suppose $h$ is concentrated on scale $r$ wave packets in $\mathbb{T}_{\tilde \theta, w}\cap \mathbb{T}_Z$. Then 
	\beq
	\int |\hat{h}|^2\cdot  \chi_{N_{C\rho^{1/2+\delta_m}}(\Phi(Z_0)\cap \Phi(CB))}\lesssim r^{O(\delta_m)}\Big(\frac{\rho}{r}\Big)^{(n-m)/2}\|h\|_2^2.
	\eeq
\end{proposition}

 Define $T_{Z,B,\tilde \theta}$ as follows 
\beqq
\mathbb{T}_{V,B,\tilde \theta}:=\{(\theta,v):T_{\theta,v}\cap B\neq \emptyset,\measuredangle(G(\xi_\theta),V)\lesssim r^{-1/2+\delta_m},\; {\rm dist}(\theta,\tilde \theta)\lesssim \rho^{-1/2}\}.
\eeqq
Let $V$ be an $m$-dimensional subspace defined by 
\beqq
V:=\{x\in \R^n: \sum_{j=1}^na_{i,j}x_j=0, i=1,\cdots, n-m\}.
\eeqq
If  $\phi(x,\xi)=x'\cdot \xi+\frac{1}{2}x_n|\xi|^2$, then 
\beqq
G(\xi)=\frac{(-\xi,1)}{\sqrt{1+|\xi|^2}}.
\eeqq
It is easy to see that $ \{\xi\in \R^n : G(\xi)\in V\} $
defines an affine subspace. Therefore, the set 
$$\{\xi\in \R^n:\measuredangle(G(\xi),V)\lesssim r^{-1/2+\delta_m}\}$$
is contained in a $Cr^{-1/2+\delta_m}$-neighborhood of an affine subspace. However, if we only know  $\phi(x,\xi)$ satisfies the Carleson-Sj\"olin conditions with the  convex and straight assumptions, things may be a little trickier. Since in the general setting, $\{\xi\in \R^n: G(\xi)\in V\}$ may be a curved submanifold. 
Specially, for our case,
 $$G(\xi)=\frac{(-\partial_\xi h(\xi),1)}{\sqrt{1+|\partial_\xi h(\xi)|^2}},$$
 where $h(\xi)$ is defined in \eqref{eq:hxi} and satisfies \eqref{eq:hass}.
 Let $L$ denote the submanifold 
 \beqq
 L:=\{\xi\in \R^n: G(\xi)\in V\},
  \eeqq
  by the implicit function  theorem, we know the dimension of $L$ is $m-1$. Define $V'$ to be the tangent space $T_{\tilde \xi } L$ of $L$ at a given point $\tilde\xi \in L$ with ${\rm dist}(\tilde{\xi},\tilde \theta)\lesssim \rho^{-1/2}$.
\begin{lemma}
	The set
\beqq
\{\xi \in \R^n:\measuredangle(G(\xi),V)\lesssim r^{-1/2+\delta_m},\;{\rm dist}(\xi,\tilde \theta)\lesssim \rho^{-1/2}\}
\eeqq
is contained in a $Cr^{-1/2+\delta_m}$-neighborhood of an affine subspace $V'$.
\end{lemma}
\begin{proof}
Obviously, the set 
\beqq
\{\xi\in \R^n: \measuredangle(G(\xi),V)\lesssim r^{-1/2+\delta_m}\}
\eeqq	
is contained in a  $Cr^{-1/2+\delta_m}$-neighborhood of  $L$.
Recall that 
$$V'=T_{\tilde \xi}L,\;\text{and}\;{\rm dist}(\tilde \xi, \tilde \theta)\lesssim r^{-1/2+\delta_m},$$
thus $${\rm dist}(\tilde \xi,\xi)\lesssim \rho^{-1/2}.$$	
Therefore, it suffices to show 
\beqq
N_{Cr^{-1/2+\delta_m}}(L)\cap \{\xi: {\rm dist}(\xi,\tilde \xi)\lesssim \rho^{-1/2}\}\subset N_{Cr^{-1/2+\delta_m}}(V').
\eeqq	
Without loss of generality, we may assume $\tilde \xi$  and $L$ can be  parametrized by $(0,u(0))$ and  
\beqq
L:=\{(\xi',u(\xi')): \xi'\in \R^{m-1}\}, \quad\text{ with}\quad  u'(0)=0,
\eeqq
respectively. Therefore, it remains to show : if $|\xi'|\lesssim \rho^{-1/2}$
\beq\label{eq:245}
|u(\xi')|\lesssim r^{-1/2+\delta_m}.
\eeq
Indeed, \eqref{eq:245} can be easily obtained from Taylor's formula and  the fact the second order derivatives of $u$ can be uniformly bounded.
	\end{proof}
Define $\widetilde{V}$ to be the orthogonal complement in $\R^{n-1}$, that is 
$\widetilde{V}:=(V')^{\perp},$
and $\bar{V}\subset \R^{n-1}$ to be identified with $V\cap \{x_n=0\}$.
\begin{lemma}
Let $V$ and $\widetilde{V}$ be defined as above, then $\widetilde{V}$ and $V$ are transverse in the sense that 
\beq\label{eq:122}
\mathop{{\rm Angle}}_{v\in \bar {V}\backslash\{0\},\tilde v\in \widetilde{V}\backslash\{0\}}(v,\tilde v)\gtrsim 1.
\eeq
\end{lemma}
\begin{proof}
Since $G(\tilde \xi)\in V$, we may write it explicitly as follows
\beq\label{eq:121}
\sum_{j=1}^{n-1}a_{i,j}\partial_{\xi_j}h(\tilde \xi)-a_{i,n}=0,\; i=1,\cdots, n-m.
\eeq
Define $\alpha_i:=(a_{i,1},\cdots,a_{i,n})$ and  $\alpha_i':=(a_{i,1},\cdots,a_{i,n-1})$.
From \eqref{eq:121},  we have 
\beqq
{\rm rank}(\alpha_1',\cdots \alpha_{n-m}')=n-m.
\eeqq
Since $V'=T_{\tilde \xi} L$,  and $\widetilde{V}=(V')^{\perp}$,  by \eqref{eq:121}, we have
\beqq
\widetilde{V}={\rm span}\langle\partial_{\xi\xi}^2 h(\tilde \xi )\alpha_1',\cdots, \partial_{\xi\xi}^2h(\tilde \xi ) \alpha_{n-m}'\rangle.
\eeqq
To prove \eqref{eq:122}, it suffices to show: for each $\bar v \in \bar V\backslash \{0\}$, then \beq\label{eq:123}  \langle \partial_{\xi\xi}^2 h(\tilde \xi)\alpha_i', \bar v\rangle \ll 1.
\eeq
Since $$\langle \alpha_i', \bar v\rangle=0 $$
and 
\beqq
\|\partial_{\xi\xi}^2 h(\tilde \xi)-I_{n-1}\|_{\rm op}\leq \varepsilon_0,
\eeqq
thus \eqref{eq:123} follows immediately by choosing $\varepsilon_0$ sufficiently small.
\end{proof}
\begin{lemma}\cite{Guth18}\label{local}
	Suppose that $G:\R^n\rightarrow \mathbb{C}$ is a function, and $\hat{G}$ is supported in a ball $B_r^n(\xi_0)$ of radius $r$. Then, for any ball $B_\rho^n(x_0)$ of radius $\rho\leq r^{-1}$,
	\beq
	\int_{B_\rho^n(x_0)}|G|^2\lesssim \frac{|B_\rho^n|}{|B_{r^{-1}}^n|}\int |G|^2.
	\eeq
	\end{lemma}
Finally, as a consequence of the above preparations, we have
\begin{proposition}\label{prob}
	Let $V,\bar V,\widetilde V,V'$ be defined as above. If $g$ is concentrated on wave packets from $\mathbb{T}_{V,B,\tilde \theta}$, if $\Pi\subset \{x_n=\tilde x_n^0\}$  is any affine subspace parallel to $\widetilde V$ and $y\in \Pi\cap \Phi(CB)$, then 
	\beqq
	\int_{\Pi \cap B_{\rho^{1/2+\delta_m}}^n(y)}|\hat{g}|^2 \lesssim r^{O(\delta_m)}\big(\frac{\rho^{1/2}}{r^{1/2}}\big)^{{\rm dim}(\widetilde{V})}\int_{\Pi}|\hat{g}|^2.
	\eeqq
\end{proposition}
\begin{proof}
Note that $g$ is concentrated from $\mathbb{T}_{V,B,\tilde \theta}$,	from the above discussion, we have $g$ is supported in the $r^{-1/2+\delta_m}$ neighborhood of $V'$. Thus, $\big(\hat{g}|_{\Pi}\big)^{\vee}$ is supported in an $n-m$ dimensional $r^{-1/2+\delta_m}$ ball centered at ${\rm proj}_{\widetilde V}(\xi_V)$, by Lemma \ref{local}, we have 
\beqq
|\hat{g}|_{\Pi}|\lesssim |\hat{g}|_\Pi|\ast \eta_{r^{1/2-\delta_m}}.
\eeqq
Then, we may integrate $|\hat{g}|_{\Pi}|$ inside the ball $B_{\rho^{1/2+\delta_m}}^n(y)$ and use H\"older's inequality to obtain the desired results.
	\end{proof}
Before the proof Proposition \ref{pro:2}, we still  needs some additional inputs.

\begin{proposition}\cite{GOW}\label{proc}
	\begin{itemize}
		\item[1:] $\Phi(Z_0)$ is quantitatively transverse to $\widetilde V$ at every point $z\in \Phi(Z_0)\cap \Phi(CB)$.
		\item[2:] $\Phi^{-1}(\Pi)$ is an $n-1-{\rm dim}(V')$ dimensional transverse complete intersection in $\R^{n-1}$.
		\item[3:] $\Pi\cap N_{C\rho^{1/2+\delta_m}}(\Phi(Z_0)\cap \Phi(CB))$ can be covered by $\big(\frac{r^{1/2}}{\rho^{1/2}}\big)^{{\rm dim}Z_0-{\rm dim}V_0}$ many balls in $\Pi$ of radius $\rho^{1/2+\delta_m}$.
		\end{itemize}
	\end{proposition}

{\bf Proof of Proposition \ref{pro:2}:}

Since  the wave packets in $\mathbb{T}_{Z,B,\tilde \theta}$ are tangent to $Z$ in $B$, thus 
\beqq
\measuredangle(G(\theta), T_zZ)\lesssim r^{-1/2+\delta_m}
\eeqq
for every $z\in Z\cap 2B$ and $T_{\theta,v}\in \mathbb{T}_{Z,B,\tilde \theta}$. There is a subspace $V$ of minimal dimension and ${\rm dim}V\leq {\rm dim }Z$ such that for all $\theta$ making contribution to $\mathbb{T}_{Z,B,\tilde \theta}$, we have 
\beqq
\measuredangle(G(\theta),V)\lesssim r^{-1/2+\delta_m},
\eeqq
which indicates that $h$ is concentrated on wave packets from $T_{V,B,\tilde \theta}$. Therefore, by Lemma \ref{prob}, we have 
\beqq
\int_{\Pi \cap B_{\rho^{1/2+\delta_m}}^n(y)}|\hat{h}|^2 \lesssim r^{O(\delta_m)}\big(\frac{\rho^{1/2}}{r^{1/2}}\big)^{{\rm dim}(\widetilde{V})}\int_{\Pi}|\hat{h}|^2.
\eeqq
Finally, by Proposition \ref{proc}, we get
	\beq
\int |\hat{h}|^2\cdot  \chi_{N_{C\rho^{1/2+\delta_m}}(\Phi(Z_0)\cap \Phi(CB))}\lesssim r^{O(\delta_m)}\big(\frac{\rho}{r}\big)^{(n-m)/2}\int_{\Pi}|\hat{h}|^2.
\eeq
Integrate over the affine subspace $\Pi$ which is parallel to $\widetilde{V}$, we will obtain the desired results.

\section{Broad-norm estimate}\label{broad}
In this section, we assume the operator $T^\lambda$ satisfies the straight condition. In this setting,  the tubes introduced in \eqref{wave-1} is straight.  To prove Theorem \ref{theo}, we will use the broad-narrow analysis developed by Bourgain-Guth \cite{BG}, which deduces the linear estimates from the multilinear ones.  In \cite{Guth}, Guth observed that full power of the $k-$linear inequality could  be replaced by a certain weakened version of the multilinear estimate for the Fourier extension operators known as $k-$broad ``norm" estimates. Following the  approach developed  by Guth in \cite{Guth}, we shall divide  $T^\lambda f$
into narrow and broad parts in the frequency space, and one part is around a neighborhood of $(k-1)$-dimensional subspace,
another comes from its outside.
We estimate the contribution of the first part through the decoupling theorem and an induction on scales argument,  and then use the $k$-broad ``norm" estimates to handle the broad part. 

First, we shall introduce a notion of broad ``norm".  Let $V\subset \R^{n}$ be a $(k-1)$-dimensional subspace.  We denote by  $\measuredangle(G(\tau), V)$  the smallest angle between the non-zero vectors $v\in V$ and $v'\in G(\tau)$.
Define $$ f_{\tau}:=f\chi_\tau.$$
For each ball $B_{K^2}^{n}\subset B_{R}^{n}$, define
\beqq
\mu_{T^\lambda}(B_{K^2}^{n}):=\min \limits_{V_1,\ldots, V_L}\max\limits_{\tau \notin V_\ell }\Big( \int_{B_{K^2}^{n}} |T^\lambda f_\tau|^p d x\Big),
\eeqq
where $\tau \notin V_\ell$ means that for all $1\leq \ell \leq L$, ${\rm Ang}(G(\tau),V_{\ell})>K^{-1}$.

Let  $\{B_{K^2}^{n}\}$ be a collection of finitely overlapping balls which form a cover of $B_{R}^n$. We define the $k$-broad ``norm" by
\beqq
\big\|T^\lambda f\big\|_{{\rm BL}_{k,L}^p(B_{R}^n)}^p:=\sum_{B^{n}_{K^2}\subset B_{R}^n} \mu_{T^\lambda }(B^{n}_{K^2}).
\eeqq
We will establish the following broad norm estimate.
\begin{theorem}\label{theo5}
	Let  $\mathscr{T}^\lambda$ be defined with $\phi$ satisfying the conditions ${\rm H}_1,{\rm H}_2,{\rm H}_3,{\rm H}_4$ and 
	\begin{itemize}
		\item The eigenvalues of the Hessian $$\partial_{\xi\xi}\langle \partial_x\phi(x,\xi ),G(x,\xi_0)\rangle|_{\xi=\xi_0}$$ all fall into $[1/2,2]$ for $x\in X, \xi_0\in \Omega$.
		\item Let  $N_{\rm par}>0$ be a given large constant as above, 
		\beqq
		|\partial_x^\alpha \partial_\xi^\beta \phi(x,\xi)|\leq C_{\alpha,\beta},\quad |\alpha|,|\beta|\leq N_{\rm par}.
		\eeqq
	\end{itemize}
	 If $2\leq k\leq n-1$ and 
	\beq
	p\geq p_n(k):=2+\frac{6}{2(n-1)+(k-1)\prod_{i=k}^{n-1}\frac{2i}{2i+1}},
	\eeq
	then for every $\varepsilon>0$, there exits $L$ such that 
	\beq\label{eq-22}
	\big\|\mathscr{T}^\lambda f\big\|_{{\rm BL}_{k,L}^p(B^{n}_R(0))}\lesssim_{\varepsilon,L,K}R^\varepsilon \|f\|_{L^2}^{2/p}\|f\|_{L^\infty}^{1-2/p},
	\eeq
	for every $K\geq 1, 1\leq R\leq \lambda$. Furthermore, the implicit constant depends polynomially on $K$. 
\end{theorem}
By combining the material in Section \ref{wav},\ref{tran} and using the polynomial partitioning method,  we may obtain the proof of Theorem \eqref{theo5}.
At this point, there is no difference between our case and that in \cite{GOW}. Therefore, one may refer to \cite{GOW} for details.
 \section{Going from k-broad to linear estimates}\label{fp}
 \begin{proposition}{\label{prop reduction}}
 	Let $T^\lambda$ be defined with the reduced form and $\mathscr{T}^\lambda$ be defined as above. Suppose that for all $K\geqslant1,\varepsilon>0$, the operator $\mathscr{T}^\lambda$ obeys the $k$-broad inequality
 	\beq
 	\|\mathscr{T}^\lambda f\|_{{\rm BL}_{k,L}^p(B_R^n(0))}\lesssim_{K, \varepsilon,L}R^\varepsilon\|f\|_{L^2}^{\frac{2}{p}}\|f\|_{L^\infty}^{1-\frac{2}{p}},
 	\eeq
 	for some fixed $k,p,L$ and all $R\geq 1$. If
 	\beqq
 	2\frac{2n-k+2}{2n-k}< p\leqslant 2\frac{k-1}{k-2},
 	\eeqq
 	then 
 	\beqq
 	\|T^\lambda f\|_{L^p(B_R^n(0))}\lesssim_{\varepsilon}R^\varepsilon \|f\|_{L^2}^{\frac{2}{p}}\|f\|_{L^\infty}^{1-\frac{2}{p}}.
 	\eeqq
 \end{proposition}
Define  $p_n$  as follows  
 \beqq
 p_n:=\min_{2\leq k\leq n-1} \max \Big\{2\frac{2n-k+2}{2n-k},2+\frac{6}{2(n-1)+(k-1)\prod_{i=k}^{n-1}\frac{2i}{2i+1}}\Big\}.
 \eeqq
 Therefore, as a consequence of Theorem \ref{theo5} and Proposition \ref{prop reduction}, Theorem \ref{theo} holds for all $p>p_n$.
 For a given dimension, we can find $k$ so that $p_n$ achieves the smallest value through an effort of calculation.
 However, we can not give  a compact formula for the explicit range of $p_n$ as mentioned in \cite{HZ} for general dimensions.

To prove Proposition \ref{prop reduction}, we also need the decoupling inequality.
\begin{lemma}[Decoupling inequality]{\label{decoupling}}
	Let $T^\lambda$ be a H\"ormander-type operator with the convex condition and $V\subset\mathbb{R}^{n}$ be an m-dimensional linear subspace, then for $2\leqslant p\leqslant 2m/(m-1)$ and $\delta>0$, we have
	\beqq
	\big\|\sum_{\tau\in V}T^\lambda g_\tau\big\|_{L^p(B_{K^2}^n)}\lesssim_\delta K^{(m-1)(1/2-1/p)+\delta}\left(\sum_{\tau\in V}\|T^\lambda g_\tau\|_{L^p(w_{B_{K^2}^n})}^p\right)^{1/p}.
	\eeqq
	Here, the sum over all caps $\tau$ for which $\measuredangle(G(\tau),V)\leqslant K^{-1}$.
\end{lemma}
Heuristically, if $K^2 \leq \lambda ^{1/2-\delta}$ with $0<\delta<1/2$, $T^\lambda $ is essentially equivalent to the translation invariant case on $B_{K^2}^n$, the fact can be seen by expanding the phase using Taylor's formula. Then Lemma \ref{decoupling} can be obtained directly by using the  sharp $\ell^2$-decoupling theorem of  Bourgain-Demeter \cite{BD} and H\"older's inequality. For more details, One may refer to \cite{BHS,ILX}.
 
 \begin{lemma}\label{sds}
 	Let $\mathcal{D}$ is a maximal $R^{-1}$-separated discrete subset of $\Omega$, then
 	\beq\label{sds1}
 	\left\|\sum_{\xi_\theta\in\mathcal{D}}e^{2\pi i\phi^\lambda(\cdot,\xi_\theta)}F(\xi_\theta)\right\|_{L^p(B_R^n(0))}\lesssim Q_p(\lambda, R)R^{(n-1)/p^\prime} \|F\|_{l^2(\mathcal{D})}^{\frac{2}{p}}\|F\|_{l^\infty(\mathcal{D})}^{1-\frac{2}{p}}\eeq for all $F:\mathcal{D}\rightarrow\mathbb{C}$, where 
 	\beqq
 	\|F\|_{\ell^p(\mathcal{D})}:=\big(\sum_{\xi_\theta \in \mathcal{D}}|F(\xi_\theta)|^p\big)^{\frac{1}{p}},
 	\eeqq
 	for $1\leq p<\infty$ and $p=\infty$ with a usual modification. \end{lemma}
 \begin{proof}
 	Here our proof is essentially the same as that of Lemma 11.8  in \cite{GHI}. Let $\eta$ be a bump smooth function on $\mathbb{R}^{n-1}$, which is  supported on $B_2^{n-1}(0)$ and equals to $1$ on $B_1^{n-1}(0)$. For each $\xi_\theta\in\mathcal{D}$, we set $\eta_{\theta}(\xi):=\eta(10R(\xi-\xi_\theta)).$% Then for each $x\in B(0,R)$, the sum on the left-side of \eqref{sds1} is a constant multiple of
 	% \beq\label{sds2}
 	% R^{n-1}\int_{\mathbb{R}^{n-1}}e^{2\pi i\phi^\lambda(x;\omega)}a^\lambda(x;\omega)\left[\tilde{\psi}(x/R)\sum_{\xi_\theta\in\mathcal{D}}e^{2\pi i\lambda\xi_\theta(x/\lambda;\omega)}F(\xi_\theta)\psi_\theta(\omega)\right]d\omega,\eeq
 	%where $\xi_\theta(x;\omega):=\phi(x;\omega)-\phi(x;\xi_\theta),\tilde{\psi}$ is a function of $n$ variables which enjoys properties similar to $\psi$ and $a^\lambda$ is a suitable choice of amplitude. Note that 
 	% \beqq
 	%\sup_{\omega\in\text{supp} \psi_\theta}|\partial_x^\beta\xi_\theta(x;\omega)|\lesssim_\beta R^{-1}|x|\quad \text {for all} \quad \beta\in\mathbb{N}_0^n \quad\text{and} x\in X. \eeqq
 	Then as in Lemma 11.8 of \cite{GHI}, we have
 	\beq\label{sds2}
 	\left| \sum_{\xi_\theta\in\mathcal{D}}e^{2\pi i\phi^\lambda(\cdot;\xi_\theta)}F(\xi_\theta)\right|\lesssim R^{n-1}\sum_{k\in\mathbb{Z}^n}(1+|k|)^{-(n+1)}|T^\lambda f_k(x)|,\eeq
 	where $T^\lambda$ is defined with the reduced form and
 	\beqq f_k(\xi ):=\sum_{\xi_\theta\in\mathcal{D}}F(\xi_\theta)c_{k,\theta}(\xi)\eta_\theta(\xi)\eeqq
 	with  $\|c_{k,\theta}(\xi)\|_\infty\leq1$. By the definition of $Q_p(\lambda, R)$ and \eqref{sds2}, 
 	\beqq
 	\left\|\sum_{\xi_\theta\in\mathcal{D}}e^{2\pi i\phi^\lambda(\cdot;\xi_\theta)}F(\xi_\theta)\right\|_{L^p(B_R^n(0))}\lesssim Q_p(\lambda, R)R^{n-1}\sum_{k\in\mathbb{Z}^n}(1+|k|)^{-(n+1)}\|f_k\|_{L^2(B^{n-1}_2)}^{\frac{2}{p}}\|f_k\|_{L^\infty(B^{n-1}_2)}^{1-\frac{2}{p}}\eeqq
 	The support of $\eta_\theta$ are pairwise disjoint, for any $q>0$, we have
 	\beqq
 	\|f_k\|_{L^q(B^{n-1}_2)}\lesssim R^{-(n-1)/q} \|F\|_{l^q(\mathcal{D})}.\eeqq
 	Thus we get
 	\begin{align*}
 		\left\|\sum_{\xi_\theta\in\mathcal{D}}e^{2\pi i\phi^\lambda(\cdot;\xi_\theta)}F(\xi_\theta)\right\|_{L^p(B_R^n(0))}&\lesssim Q_p(\lambda,R)R^{n-1}\sum_{k\in\mathbb{Z}^n}(1+|k|)^{-(n+1)}R^{-(n-1)/p} \|F\|_{l^2(\mathcal{D})}^{\frac{2}{p}}\|F\|_{l^\infty(\mathcal{D})}^{1-\frac{2}{p}}\\
 		&\lesssim Q_p(\lambda, R)R^{(n-1)/p^\prime} \|F\|_{l^2(\mathcal{D})}^{\frac{2}{p}}\|F\|_{l^\infty(\mathcal{D})}^{1-\frac{2}{p}}.
 	\end{align*}
 \end{proof}
 
 \begin{lemma}{\label{rescaling}}{(Parabolic rescaling)}
 	Let $1\leqslant R\leqslant \lambda$, and $f$ supported in  a ball of radius $K^{-1}$, where $1\leqslant K\leqslant R$. Then for all $p\geqslant2$ and $\delta>0$, we have
 	\beqq
 	\|T^\lambda f\|_{L^p(B_R^n(0))}\lesssim_\delta Q_p\Big(\frac{\lambda}{K^2},\frac{R}{K^2}\Big)R^\delta K^{2n/p-(n-1)}\|f\|_{L^2(B^{n-1}_1)}^{\frac{2}{p}}\|f\|_{L^\infty(B^{n-1}_1)}^{1-\frac{2}{p}}.
 	\eeqq
 \end{lemma}
 \begin{proof}
 	Without loss of generality, we may assume the ball to be $B_{K^{-1}}^{n-1}(\bar \xi)$.
 	Doing the same argument as in Section \ref{reduc}, we obtain
 	\beqq
 	\|T^\lambda f\|_{L^p(B_R^n(0))}\lesssim_\delta K^{(n+1)/p}\|\widetilde{T}^{\lambda/K^2}\tilde{f}\|_{L^p(\tilde{D}_R)}
 	\eeqq
 	where $\widetilde{T}^{\lambda/K^2}$ is defined with phase $\tilde{\phi}$ as in \eqref{eq:213} and $\tilde{D}_R$ is an ellipse with principle axes parallel to the coordinate axes and dimensions $O(R/K)\times...\times O(R/K) \times O(R/K^2)$
 	and $\tilde{f}(\xi):=K^{-(n-1)}f(\bar{\xi}+K^{-1}\xi)$, note that for each $q>0$,
 	\beqq
 	\|\widetilde{f}\|_{L^q}\lesssim K^{-(n-1)+(n-1)/q}\|f\|_{L^q}.\eeqq
 	Then it suffice to show that
 	\beqq \|\widetilde{T}^{\lambda/K^2}\tilde{f}\|_{L^p(\tilde{D}_R)}\lesssim_\delta Q_p\Big(\frac{\lambda}{K^2},\frac{R}{K^2}\Big)R^\delta \|\tilde{f}\|_{L^2(B^{n-1}_1)}^{\frac{2}{p}}\|\tilde{f}\|_{L^\infty(B^{n-1}_1)}^{1-\frac{2}{p}}.\eeqq
 	Since the phase $\tilde{\phi}$ is also of reduced form, to ease notations, we just need to show 
 	\beqq \|T^{\lambda}f\|_{L^p(D_R)}\lesssim_\delta Q_p(\lambda, R)R^\delta \|f\|_{L^2(B^{n-1}_1)}^{\frac{2}{p}}\|f\|_{L^\infty(B^{n-1}_1)}^{1-\frac{2}{p}}.\eeqq
 	for all $1\ll R\leq R^\prime\leq \lambda$ and $\delta>0$, where
 	$$D_R:=\left\{x\in\mathbb{R}^n:\left(\frac{|x^\prime|}{R^\prime}\right)^2+\left(\frac{|x_n|}{R}\right)^2\leq1\right\}$$
 	is an ellipse and $T^\lambda$ is an operator with the reduced form.  Choose a collection of essentially disjoint $R^{-1}$-caps $\theta$  covers $B^{n-1}$, denote the center of $\theta$ by $\xi_\theta$ and  decompose $f$ as $f=\sum_\theta f_\theta$. Set
 	\beqq
 	T^\lambda_\theta f(x):=e^{-2\pi i\phi^\lambda(x,\xi_\theta)}T^\lambda(x),\eeqq
 	hence we have
 	\beqq
 	T^\lambda f(x)=\sum_\theta e^{-2\pi i\phi^\lambda(x,\xi_\theta)}T^\lambda_\theta f_\theta(x).\eeqq
 	Fix $\delta>0$ to be sufficiently small for the purpose of the forthcoming argument. 
 	We may also write
 	\beqq
 	T^\lambda_\theta f_\theta(x)=T^\lambda_\theta f_\theta*\eta_{R^{1-\delta}}(x)+{\rm RapDec}(R)\|f\|_{L^2(B^{n-1})}\eeqq
 	for some choice of smooth, rapidly decreasing function $\eta$ such that $|\eta|$ admits a smooth rapidly decreasing majorant $\zeta:\mathbb{R}^n\rightarrow[0,+\infty)$ which is locally constant at scale $1$. In particular, it follows that
 	\beq\label{zeta}
 	\zeta_{R^{1-\delta}}(x)\lesssim R^\delta \zeta_{R^{1-\delta}}(x) \quad \text{if}\quad |x-y|\lesssim R.\eeq
 	
 	Cover $D_R$ by finitely-overlapping $R$-balls, and let $B_R^n$ be some member of this cover with the center denoted by $\bar{x}$ , by the above observation, for $z\in B_R^n(0)$, we have
 	\beqq
 	|T^\lambda f(\tilde{x}+z)|\lesssim R^\delta\int_{\mathbb{R}^n}\left|\sum_{\theta} e^{2\pi i\phi^\lambda(\bar{x}+z,\xi_\theta)}T^\lambda_\theta f_\theta(y)\right| \zeta_{R^{1-\delta}}(\bar{x}-y)dy.\eeqq
 	By taking the $L^p$-norm in $z$ and modifying  the proof of Lemma \ref{sds} for the phase $\phi^\lambda(\bar{x}+\cdot,\xi_\theta)$, we have
 	\begin{align*}
 		\|T^\lambda f\|_{L^p(B_R^n(0))}& \lesssim  R^\delta\int_{\mathbb{R}^n}\left\|\sum_{\theta} e^{2\pi i\phi^\lambda(\bar{x}+z,\xi_\theta)}T^\lambda_\theta f_\theta(y)\right\|_{L^p(B_R^n(0))} \zeta_{R^{1-\delta}}(\bar{x}-y)dy \\
 		& \lesssim Q_p(\lambda,R) R^{(n-1)/p^\prime} R^\delta\int_{\mathbb{R}^n}\|T^\lambda_\theta f_\theta(y)\|_{l^2(\theta)}^{2/p}\|T^\lambda_\theta f_\theta(y)\|_{l^\infty(\theta)}^{1-2/p}\zeta_{R^{1-\delta}}(\bar{x}-y)dy,
 	\end{align*}
 where we use $\|a_{\theta}\|_{\ell^p(\theta)}$ to denote  $\big(\sum\limits_{\theta}|a_\theta|^p\big)^{1/p}$.
 
 	By property \eqref{zeta}, for $z\in B_R^n(0)$
 	\begin{align*}
 		&\int_{\mathbb{R}^n}\|T^\lambda_\theta f_\theta(y)\|_{l^2(\theta)}^{2/p}\|T^\lambda_\theta f_\theta(y)\|_{l^\infty(\theta)}^{1-2/p}\zeta_{R^{1-\delta}}(\bar{x}-y)dy\\
 		=&\int_{\mathbb{R}^n}\|T^\lambda_\theta f_\theta(\bar{x}+z-y)\|_{l^2(\theta)}^{2/p}\|T^\lambda_\theta f_\theta(\bar{x}+z-y)\|_{l^\infty(\theta)}^{1-2/p}\zeta_{R^{1-\delta}}(y-z)dy \\
 		\lesssim &R^{O(\delta)}\int_{\mathbb{R}^n}\|T^\lambda_\theta f_\theta(\bar{x}+z-y)\|_{l^2(\theta)}^{2/p}\|T^\lambda_\theta f_\theta(\bar{x}+z-y)\|_{l^\infty(\theta)}^{1-2/p}\zeta_{R^{1-\delta}}(y)dy\\
 		\lesssim& R^{O(\delta)}\left(\int_{\mathbb{R}^n}\|T^\lambda_\theta f_\theta(\bar{x}+z-y)\|_{l^2(\theta)}^{2}\|T^\lambda_\theta f_\theta(\bar{x}+z-y)\|_{l^\infty(\theta)}^{p-2}\zeta_{R^{1-\delta}}(y)dy\right)^{1/p}
 	\end{align*}
 	Then we deduces that for all $z\in B_R^n(0)$
 	\beqq
 	\begin{aligned}
 	\quad \|T^\lambda f\|_{L^p(B_R^n(0))}&\lesssim Q_p(\lambda, R) R^{(n-1)/p^\prime}  R^{O(\delta)}\\ &\times \left(\int_{\mathbb{R}^n}\|T^\lambda_\theta f_\theta(\bar{x}+z-y)\|_{l^2(\theta)}^{2}\|T^\lambda_\theta f_\theta(\bar{x}+z-y)\|_{l^\infty(\theta)}^{p-2}\zeta_{R^{1-\delta}}(y)dy\right)^{1/p}
 	\end{aligned}\eeqq
 	By raising both sides of this estimate to the $p$th power, averaging in $z$ and summing over all balls $B_R^n(0)$ in the covering, it follows 
 	that $\|T^\lambda f\|_{L^p(D_R)}$ is dominated by
 	\beqq
 	Q_p(\lambda,R) R^{(n-1)/p^\prime-n/p} R^{O(\delta)}\left(\int_{\mathbb{R}^n}\sum_\theta\|T^\lambda_\theta f_\theta\|_{L^2(D_R-y)}^{2}\sup_\theta\|T^\lambda_\theta f_\theta\|_{L^\infty(D_R-y)}^{p-2}\zeta_{R^{1-\delta}}(y)dy\right)^{1/p}\eeqq
 	We have the trivial estimate
 	\beqq
 	\|T^\lambda_\theta f_\theta\|_{L^\infty(D_R-y)}\lesssim\|f_\theta\|_{L^1}\lesssim R^{-(n-1)}\|f_\theta\|_{L^\infty}\eeqq
 	and 
 	\beqq
 	\|T^\lambda_\theta f_\theta\|_{L^2(D_R-y)}\lesssim R^{1/2}\|f_\theta\|_{L^2}.\eeqq
 	Hence  $\|T^\lambda f\|_{L^p(D_R)}$ is dominated by $  Q_p(\lambda,R) R^{O(\delta)}\|f\|_{L^2}^{\frac{2}{p}}\|f\|_{L^\infty}^{1-\frac{2}{p}}.$
 \end{proof}

 {\bf \noindent Proof of Proposition \ref{prop reduction}:}
 Let $T^\lambda$ be defined with the reduced form. Then there exits a smooth function $p(x)$ such that 
 $\phi(p(x), \xi)$ satisfies the straight condition with  \eqref{eq:uni1} holding. For convenience, we denote $\bar \phi(x,\xi):=\phi(p(x),\xi)$ and $\mathscr{T}^\lambda$ be defined with $\bar \phi$.  Hence, we have 
 \beqq
 \|T^\lambda f\|_{L^p(B_R^n(0))}\lesssim \|\mathscr{T}^\lambda f\|_{L^p(B_{CR}^n(0))}.
 \eeqq
 For a given ball $B_{K^2}^n$, we chose a collection of $(k-1)$- subspaces $V_1, ..., V_L$ which 
 achieve the minimum under the definition of $k$- board ``norm". Then
 \beqq
 \int_{B_{K^2}^n}|\mathscr{T}^\lambda f|^p\lesssim K^{O(1)}\max_{\tau\notin V_\ell , 1\leqslant \ell \leqslant L} \int_{B_{K^2}^n}|\mathscr{T}^\lambda f_\tau|^p+\sum_{\ell=1}^L\int_{B_{K^2}^n}|\sum_{\tau\in V_\ell}\mathscr{T}^\lambda f_\tau|^p.
 \eeqq
 We can use the $k$-broad hypothesis to dominate the first term, indeed, let $\mathcal{B}_{K^2}$ be a collection of finitely overlapping
 balls of radius $K^2$ which cover $B_{CR}^n(0)$, then one has
 \begin{align*}
 	\int_{B_{CR}^n(0)}|\mathscr{T}^\lambda f|^p & \lesssim\ K^{O(1)}\|\mathscr{T}^\lambda f\|_{{\rm BL}_{k,L}^p(B_{CR}^n(0))}^p+\sum_{B_{K^2}^n \in\mathcal{B}_{K^2}}\sum_{\ell=1}^L\int_{B_{K^2}^n}|\sum_{\tau\in V_\ell}\mathscr{T}^\lambda f_\tau|^p\\&
 	\lesssim K^{O(1)}C(K,\varepsilon,L)R^{p\varepsilon}\|f\|_{L^p}^p +\sum_{B_{K^2}^n \in\mathcal{B}_{K^2}^n }\sum_{\ell=1}^L\int_{B_{K^2}^n}|\sum_{\tau\in V_\ell}\mathscr{T}^\lambda f_\tau|^p.
 \end{align*}
 By Lemma \ref{decoupling}, for any $\delta^\prime$, we have
 \beqq
 \int_{B_{K^2}^n}|\sum_{\tau\in V_\ell}\mathscr{T}^\lambda f_\tau|^p\lesssim_{\delta^\prime}K^{(k-2)(p/2-1)+\delta^\prime}\sum_{\tau\in V_\ell}\int_{\mathbb{R}^n}|\mathscr{T}^\lambda f_\tau|^pw_{B_{K^2}^n }
 \eeqq for each $1\leqslant \ell \leqslant L$.
 Since $w_{B_R^n(0)}=\sum_{B_{K^2}^n \in\mathcal{B}_{K^2}}w_{B_{K^2}^n }$, one has
 \beqq
 \sum_{B_{K^2}^n \in\mathcal{B}_{K^2}}\sum_{\ell=1}^L\int_{B_{K^2}^n }|\sum_{\tau\in V_\ell }\mathscr{T}^\lambda f_\tau|^p\lesssim_{\delta^\prime}K^{(k-2)(p/2-1)+\delta^\prime}\sum_\tau\int_{\mathbb{R}^n}|\mathscr{T}^\lambda f_\tau|^pw_{B_{R}^n(0)}.
 \eeqq
 For each $\tau$, we take the same approach as in Section \ref{reduc} which obtains the reduced form from a general phase. To ease the notations, under the new coordinates, we use  $T^\lambda f_\tau$ to denote the new operator  which belongs to the reduced form and the new function.  Therefore, 
 \beqq
 \int_{\mathbb{R}^n}|\mathscr{T}^\lambda f_\tau|^pw_{B_{R}^n(0)}\lesssim \int_{\mathbb{R}^n} |T^\lambda f_\tau|^pw_{B_{CR^{}}^n(0)}.
 \eeqq
 Note that $w_{B_{CR}^n(0)}$ rapidly decay outside $B_{2CR}^n(0)$, we get
 \beqq
 \sum_{B_{K^2}^n \in\mathcal{B}_{K^2}}\sum_{\ell=1}^L \int_{B_{K^2}^n }|\sum_{\tau\in V_\ell }T^\lambda f_\tau|^p\lesssim_{\delta^\prime}K^{(k-2)(p/2-1)+\delta^\prime}\sum_\tau\int_{B_{2CR}^n (0)}|T^\lambda f_\tau|^p.
 \eeqq
 Let $\delta>0$ be a small number to be determined later. By a finitely-overlapping decomposition and translation,  from Lemma \ref{rescaling},  we obtain
 \beqq
 \int_{B_{2CR}^n(0)}|T^\lambda f_\tau|^p\lesssim Q_p\Big(\frac{\lambda}{K^2}, \frac{R}{K^2}\Big)R^\delta K^{2n-(n-1)p} \|f_\tau\|_{L^2}^{2}\|f_\tau\|_{L^\infty}^{p-2}.\eeqq
 Let 
 \beqq e(k,p):=(k-2)(1-\frac{1}{2}p)-2n+(n-1)p.
 \eeqq
 Recall
 \beqq
 \sum_{\tau}\|f_\tau\|_{L^2}^2\lesssim\|f\|_{L^2}^2,
 \eeqq
 therefore, we have
 \beqq
 \sum_{B_{K^2}^n\in\mathcal{B}_{K^2}^n}\sum_{\ell=1}^L\int_{B_{K^2}}|\sum_{\tau\in V_\ell}T^\lambda f_\tau|^p\lesssim_{\delta,\delta^\prime} Q_p\Big(\frac{\lambda}{K^2}, \frac{R}{K^2}\Big)^pR^\delta K^{-e(k,p)+\delta^\prime}\|f\|_{L^2}^{2}\|f\|_{L^\infty}^{p-2}.
 \eeqq
 Combining above estimates, we get
 \beqq
 \int_{B_R^n(0)}|T^\lambda f|^p\leqslant(K^{O(1)}C(K,\varepsilon,L)R^{p\varepsilon}+C_{\delta,\delta^\prime} Q_p\Big(\frac{\lambda}{K^2}, \frac{R}{K^2}\Big)^pR^\delta K^{-e(k,p)+\delta^\prime}\|f\|_{L^2}^{2}\|f\|_{L^\infty}^{p-2}.
 \eeqq
 Then by the definition of $Q_p(\lambda, R)$,
 \beqq
 Q_p(\lambda, R)^p\leqslant K^{O(1)}C(K,\varepsilon,L)R^{p\varepsilon}+C_{\delta,\delta^\prime} Q_p(\lambda, R)^pR^\delta K^{-e(k,p)+\delta^\prime}.
 \eeqq
 When $p>2\frac{2n-k+2}{2n-k}$, $e(k,p)>0$, we choose $\delta^\prime=\frac{1}{2}e(k,p),K=K_0R^{2\delta/e(k,p)}$ where $K_0$ is a large constant depending on $\varepsilon,\delta,p$ and $n$ such  that
 \beqq
 Q_p(R)^p\leqslant K^{O(1)}C(K,\varepsilon,L)R^{p\varepsilon}+\frac{1}{2}Q_p(R)^p.
 \eeqq
 Recall that $C(K,\varepsilon,L)$ depends polynomially  on $K$, then we will complete the proof by choosing  suitable $0<\delta\ll\varepsilon$ such that $Q_p(R)\lesssim_\varepsilon R^\varepsilon.$

\bibliographystyle{amsplain}

\end{document}